\documentclass[11pt,leqno]{article}
\include{psfig}
\include{epsf}

\usepackage{color}
\usepackage{amssymb} %to define $\Box$ in latex2e

\newtheorem{theorem}{Theorem}[section]
\newtheorem{lemma}{Lemma}[section]
\newtheorem{proposition}[theorem]{Proposition}

\newtheorem{definite}{Definition}[section]

\newtheorem{remark}{Remark}[section]
\newtheorem{example}{Example}[section]
\newtheorem{assumption}{Assumption}[section]

\catcode`\@=11
\renewcommand{\section}{
         \setcounter{equation}{0}
         \@startsection {section}{1}{\z@}{-3.5ex plus -1ex minus
         -.2ex}{2.3ex plus .2ex}{\normalsize\bf}
}
\renewcommand{\subsection}{
         \@startsection {subsection}{1}{\z@}{-3.5ex plus -1ex minus
         -.2ex}{2.3ex plus .2ex}{\normalsize\bf}%{\center\normalsize\bf}
}
\catcode`\@=12

\def\reals{{\rm\vrule depth0ex width.4pt\kern-.08em R}}
\def\bbbz{{\mathchoice {\hbox{$\sf\textstyle Z\kern-0.4em Z$}}
{\hbox{$\sf\textstyle Z\kern-0.4em Z$}}
{\hbox{$\sf\scriptstyle Z\kern-0.3em Z$}}
{\hbox{$\sf\scriptscriptstyle Z\kern-0.2em Z$}}}}

\newcommand{\nc}{\newcommand}
\nc{\W}{{\bf W}}
\nc{\A}{{\bf A}}
\nc{\bL}{{\bf L}}
\nc{\bH}{{\bf H}}
\nc{\C}{{\cal C}}

\def\eq#1{(\ref{e:#1})}

\def\elabel#1{\label{e:#1}}

\begin{document}
%\begin{titlepage}
\begin{center}
\Large\bf Mean-variance hedging based on an incomplete market with
external risk factors of non-Gaussian OU processes\footnote{Partial
results and graphs are briefly summarized and reported in 2012
Spring World Congress of Engineering and Technology. This enhanced
version with extension and complete proofs of results is a journal
version of the short conference report.}
\end{center}
\begin{center}
\large\bf Wanyang Dai\footnote{Supported by National Natural Science
Foundation of China with Grant No. 10971249 and Grant No. 11371010.}
\end{center}
\begin{center}
\small Department of Mathematics and State Key Laboratory of Novel
Software Technology\\
Nanjing University, Nanjing 210093, China\\
Email: nan5lu8@netra.nju.edu.cn\\
Submitted: 29 September 2014\\
Revised: 11 March 2015
\end{center}

\vskip 0.1 in
\begin{abstract}

In this paper, we prove the global risk optimality of the hedging
strategy of contingent claim, which is explicitly (or called
semi-explicitly) constructed for an incomplete financial market with
external risk factors of non-Gaussian Ornstein-Uhlenbeck (NGOU)
processes. Analytical and numerical examples are both presented to
illustrate the effectiveness of our optimal strategy. Our study
establishes the connection between our financial system and existing
general semimartingale based discussions by justifying required
conditions. More precisely, there are three steps involved. First,
we firmly prove the no-arbitrage condition to be true for our
financial market, which is used as an assumption in existing
discussions. In doing so, we explicitly construct the
square-integrable density process of the variance-optimal martingale
measure (VOMM). Second, we derive a backward stochastic differential
equation (BSDE) with jumps for the mean-value process of a given
contingent claim. The unique existence of adapted strong solution to
the BSDE is proved under suitable terminal conditions including both
European call and put options as special cases. Third, by combining
the solution of the BSDE and the VOMM, we reach the justification of
the global risk optimality for our
hedging strategy.\\

%\noindent {\em Subject classifications:}
\end{abstract}
%\end{titlepage}

\section{Introduction}

In this paper, we justify the global risk optimality of the hedging
strategy of contingent claim, which is explicitly constructed for an
incomplete market defined on some filtered probability space
$(\Omega,{\cal F},\{{\cal F}_{t}\}_{t\geq 0},P)$. The financial
market has $d+1$ primitive assets: one bond with constant interest
rate and $d$ risky assets. The price processes of the assets are
described by a generalized Black-Scholes model with coefficients
driven by the market regime caused by leverage effect, etc. The
financial market model includes the Barndorff-Nielsen $\&$ Shephard
(BNS) volatility model proposed by Barndorff-Nielsen and
Shephard~\cite{barnie:nonorn} and further studied in Benth {\em et
al.}~\cite{benkar:merpor}, Benth and
Meyer-Brandis~\cite{benmey:denpro}, Lindberg \cite{lin:newgen}, etc.
as a particular case. Our model is closely related to the one
considered in Delong and Kl$\ddot{\mbox{u}}$ppelberg
\cite{delklu:optinv}. As pointed out in Barndorff-Nielsen and
Shephard~\cite{barnie:nonorn}, these models fit real market data
quite well. Nevertheless, such models also induce incompleteness of
the financial markets, which means that it is impossible to
replicate perfectly contingent claims based on the bond and the $d$
primitive risky assets. A rule for designing a good hedging strategy
is to minimize the mean squared hedging error over the set
$\bar{\Theta}$ of all reasonable
trading strategy processes, %(see, e.g., Dai~\cite{dai:opthed}),
\begin{eqnarray}
\inf_{u\in\bar{\Theta}}E\left[(v+(u\cdot D)(T)-H)^{2}\right],
\elabel{mvhedgep}
\end{eqnarray}
where $H$ is a random variable representing the discounted payoff of
the claim, $D$ is the discounted price process of $d$ risky assets,
$v$ is the initial endowment and $T$ is the time horizon.
Mathematically speaking, one seeks to compute the orthogonal
projection of $H-v$ on the space $\bar{\Theta}$ of stochastic
integrals.

To solve the mean-variance hedging problem \eq{mvhedgep}, we
explicitly construct a trading strategy for the financial market and
justify it to be the global risk-minimizing hedging strategy by
using the following procedure.

First, we explicitly construct the square-integrable density process
of a variance-optimal martingale measure (VOMM) $Q^{*}$. As a
result, the set of equivalent (local) martingale measures with
square-integrable densities, i.e.,
\begin{eqnarray}
&&{\cal U}_{2}^{e}(D)\equiv\left\{Q\sim P:\frac{dQ}{dP}\in L^{2}(P),
D\;\mbox{is a}\;Q\mbox{-local martingale}\right\}
\elabel{equivmeasure}
\end{eqnarray}
is nonempty. Hence, our market is arbitrage-free (e.g, Delbaen and
Schachermayer~\cite{delsch:genver}). Second, we derive an BSDE with
jumps and external random factors of non-Gaussian Ornstein-Uhlenbeck
(NGOU) type for the mean value process of the option $H$ (i.e.,
$E_{Q^{*}}[H|{\cal F}_{t}]$). The unique existence of adapted
solution to the BSDE is proved under suitable terminal conditions
including both European call and put options as special cases.
Third, by combining the solution to the BSDE and the VOMM, we get
the optimal hedging strategy for our market.

The BSDE and VOMM based procedure is a mixed method of two typical
approaches in solving mean-variance hedging problem: martingale
approach stemmed from Harrison and Kreps~\cite{harkre:mararb}, and
stochastic control approach that views the problem as a
linear-quadratic control problem and employs BSDEs to describe the
solution (see, e.g., Yong and Zhou). This procedure is structured
for a general semimartingale in C\u{e}rn\'y and
Kallsen~\cite{cerkal:strgen} and explicitly (or semi-explicitly)
presented for the current market in Dai~\cite{dai:opthed}. Some
related and independent study can also be found in Jeanblanc {\em et
al.}~\cite{jeaman:meavar}. More precisely, we have the following
literature review and technical comparisons.

A closely related (local) risk minimizing problem was initially
introduced by F\"{o}llmer and Sondermann~\cite{folson:hednon} under
complete information, who also suggested an approach for the
computation of a minimizing strategy in an incomplete market by
extending the martingale approach of Harrison and
Kreps~\cite{harkre:mararb}. The basic idea of the approach was to
introduce a measure of riskiness in terms of a conditional mean
square error process where the discounted price process is a
square-integrable martingale. Furthermore, the answer to the hedging
problem is provided by the {\em Galtchouk-Kunita-Watanabe
decomposition} of the claim. Then, this concept of local-risk
minimization was further extended for the semimartingale case by
F\"{o}llmer and Schweizer~\cite{folsch:hedcon}, and
Schweizer~\cite{sch:opthed,sch:meavar}, where the minimal martingale
measure and F\"{o}llmer-Schweizer (F-S) decomposition play a central
role. Interested readers are referred to F\"{o}llmer and
Schweizer~\cite{folsch:minmar}, Schweizer~\cite{sch:guitou} for more
recent surveys about (local) risk minimization and mean-variance
hedging.

Owing to the fact that one cares about the total hedging error and
not the daily profit-loss ratios, the solution with respect to
global-risk minimization of the unconditional expected squared
hedging error presented in \eq{mvhedgep} was considered (e.g.,
surveys in Pham~\cite{pha:quahed} and Schweizer \cite{sch:guitou}).
Then, the study on global-risk minimization was further developed by
C\u{e}rn\'y and Kallsen~\cite{cerkal:strgen}, who showed that the
hedging model \eq{mvhedgep} admits a solution in a very general
class of arbitrage-free semimartingale markets where local-risk
minimization may fail to be well defined. The key point of their
approach is the introduction of the opportunity-neutral measure
$P^{*}$ that turns the dynamic asset allocation problem into a
myopic one. Furthermore, the minimal martingale measure relative to
$P^{*}$ coincides with the variance-optimal martingale measure
relative to the original probability measure $P$. Recently, to
overcome the difficulties appeared in C\u{e}rn\'y and
Kallsen~\cite{cerkal:strgen} (i.e., a process $N$ appeared in
Definition 3.12 is very hard to find and the VOMM $Q^{*}$ in
Proposition 3.13 is notoriously difficult to determine), the authors
in Jeanblanc {\em et al.}~\cite{jeaman:meavar} developed a method
via stochastic control and backward stochastic differential
equations (BSDEs) to handle the mean-variance hedging problem for
general semimartingales. Furthermore, the authors in Kallsen and
Vierthauer~\cite{kalvie:quahed} derived semi-explicit formulas for
the optimal hedging strategy and the minimal hedging error by
applying general structural results and Laplace transform
techniques. In addition to these works, some related studies in both
general theory and concrete results in specific setups for the
mean-variance hedging problem can be found in, such as,
Arai~\cite{ara:extmea}, Chan {\em et al.}~\cite{chakol:varopt},
Duffie and Richardson~\cite{dufric:meahed}, Gourieroux {\em et
al.}~\cite{goulau:meavar}, Heath {\em et al.}~\cite{heapla:comqua},
Laurent and Pham~\cite{laupha:dynpro}, and references therein.

Comparing with the above studies, our contribution of the current
research is threefold. First, we firmly prove the no-arbitrage
condition to be true for our financial market, i.e., the set defined
in \eq{equivmeasure} is nonempty. This condition is used as an
assumption for the existence of the VOMM in existing discussions
(e.g., Arai~\cite{ara:extmea}, C\u{e}rn\'y and
Kallsen~\cite{cerkal:strgen}, Chan {\em et
al.}~\cite{chakol:varopt}, Jeanblanc {\em et
al.}~\cite{jeaman:meavar}, Kallsen and
Vierthauer~\cite{kalvie:quahed}). In doing so, we explicitly (or
called semi-explicitly) construct a measure through identifying its
explicit density by the general structure presented in C\u{e}rn\'y
and Kallsen~\cite{cerkal:strgen}. Then, we justify it to be the VOMM
for our market model by proving the equivalent conditions given in
C\u{e}rn\'y and Kallsen~\cite{cerkal:meavar}. Second, in applying
our VOMM to obtain the optimal hedging strategy, we derive an BSDE
with jumps for the mean value process of the option $H$. Here, we
lift the requirements that the contingent claims are bounded (e.g.,
Heath and Schweizer~\cite{heasch:marver}, C\u{e}rn\'{y} and
Kallsen~\cite{cerkal:meavar}) or satisfy Lipschitz condition (e.g.,
Roch~\cite{roc:vissol}, Chan {\em et al.}~\cite{chakol:varopt}) to
guarantee the corresponding integral-partial differential equation
(IPDE) to have a classic or viscosity solution. Furthermore, the
unique existence of an adapted solution to our derived BSDE is
firmly proved under certain conditions while in the recent study of
Jeanblanc {\em et al.}~\cite{jeaman:meavar}, such existence of an
adapted solution to their constructed BSDE is only showed as an
equivalent condition to guarantee the existence of an optimal
strategy. More importantly, our BSDE can be solved by developing
related numerical algorithms through the given terminal option $H$
(see, e.g., Dai~\cite{dai:nummet}). Third, from the purpose of easy
applications, our discussion is based on a multivariate financial
market model, which is in contrast to existing studies (e.g.,
C\u{e}rn\'y and Kallsen~\cite{cerkal:strgen}, Chan {\em et
al.}~\cite{chakol:varopt}, Jeanblanc {\em et
al.}~\cite{jeaman:meavar}, Kallsen and
Vierthauer~\cite{kalvie:quahed}). Therefore, unlike the studies in
Hubalek {\em et al.}~\cite{hubkal:varopt} and Kallsen and
Vierthauer~\cite{kalvie:quahed}, our option $H$ is generally related
to a multivariate terminal function and hence a BSDE involved
approach is employed. Actually, whether one can extend the Laplace
transform related method developed in Hubalek {\em et
al.}~\cite{hubkal:varopt} and Kallsen and
Vierthauer~\cite{kalvie:quahed} for single-variate terminal function
to our general multivariate case is still an open problem.

Note that our study in this paper establishes the connection between
our financial system and existing general semimartingale based study
in C\u{e}rn\'y and Kallsen~\cite{cerkal:strgen} since we can
overcome the difficulties in C\u{e}rn\'y and
Kallsen~\cite{cerkal:strgen} by explicitly constructing the process
$N$ and the VOMM $Q^{*}$ as mentioned earlier. Furthermore, our
objective and discussion in this paper are different from the recent
study of Jeanblanc {\em et al.}~\cite{jeaman:meavar} since the
authors in Jeanblanc {\em et al.}~\cite{jeaman:meavar} did not aim
to derive any concrete expression. Nevertheless, interested readers
may make an attempt to extend the study in Jeanblanc {\em et
al.}~\cite{jeaman:meavar} and apply it to our financial market model
to construct the corresponding explicit results.

Finally, when the random variable $H$ in \eq{mvhedgep} is taken to
be a constant (e.g., a prescribed daily expected return), the
associated hedging problem reduces to a mean-variance portfolio
selection problem as studied in Dai~\cite{dai:meavar} by an
alternative feedback control method. In this case, the optimal
policies can be explicitly obtained by both the feedback control
method in Dai~\cite{dai:meavar} and the martingale method presented
in the current paper. In the late method, the related BSDE is a
degenerate one. From this constant option case, we can construct two
insightful examples to provide the effective comparisons between the
two methods. More precisely, our newly constructed hedging strategy
can slightly outperform the feedback control based policy. However,
the performance between the two methods is consistent in certain
sense.

The remainder of the paper is organized as follows. We formulate our
financial market model in Section~\ref{SRBM} and present our main
theorem Section~\ref{optimalhedging}. Analytical and numerical
examples are given in Section~\ref{necom}. Our main theorem is
proven in Section~\ref{hedgee}. Finally, in Section~\ref{concl}, we
conclude this paper with remarks.

\section{The Financial Market}\label{SRBM}

\subsection{The Model}

We use $(\Omega,{\cal F},P)$ to denote a fixed complete probability
space on which are defined a standard $d$-dimensional Brownian
motion $W\equiv\{W(t),t\in[0,T]\}$ with
$W(t)=(W_{1}(t),...,W_{d}(t))'$ and $h$-dimensional subordinator
$L\equiv\{L(t),t\in[0,T]\}$ with
$L(t)\equiv(L_{1}(t),...,L_{h}(t))'$ and c\`adl\`ag sample paths for
some fixed $T\in[0,\infty)$ (e.g., Applebaum~\cite{app:levpro},
Bertoin~\cite{ber:levpro}, and Sato~\cite{sat:levpro} for more
details about subordinators and L\'evy processes). The prime denotes
the corresponding transpose of a matrix or a vector. Furthermore,
$W$, $L$, and their components are assumed to be independent of each
other. For each given $\lambda=(\lambda_{1},...\lambda_{h})'>0$, we
let $L(\lambda s)=(L_{1}(\lambda_{1}s),...,L_{h}(\lambda_{h}s))'$.
Then, we suppose that there is a filtration $\{{\cal F}_{t}\}_{t\geq
0}$ related to the probability space, where ${\cal
F}_{t}\equiv\sigma\{W(s),L(\lambda s): 0\leq s\leq t\}$ for each
$t\in[0,T]$.

The financial market under consideration is a multivariate
L\'evy-driven OU type stochastic volatility model, which consists of
$d+1$ assets. One of the $d+1$ assets is risk-free, whose price
$S_{0}(t)$ is subject to the ordinary differential equation (ODE)
with constant interest rate $r\geq 0$,
\begin{eqnarray}
%\left\{\begin{array}{ll}
       dS_{0}(t)=rS_{0}(t)dt,\;\;S_{0}(0)=s_{0}>0.
       %\end{array}
%\right.
\elabel{bankaset}
\end{eqnarray}
The other $d$ assets are stocks whose vector price process
$S(t)=(S_{1}(t),...,S_{d}(t))'$ satisfies the following stochastic
differential equation (SDE) for each $t\in[0,T]$,
\begin{eqnarray}
\left\{\begin{array}{ll}
       dS(t)=\mbox{diag}(S(t^{-}))\{b(Y(t^{-}))dt+\sigma(Y(t^{-}))dW(t)\},\\
       S(0)=s>0.
       \end{array}
\right. \elabel{stockassetm}
\end{eqnarray}
Here and in the sequel, the diag($v$) denotes the $d\times d$
diagonal matrix whose entries in the main diagonal are $v_{i}$ with
$i\in\{1,...,d\}$ for a $d$-dimensional vector
$v=(v_{1},...,v_{d})'$ and all the other entries are zero. $Y(t)$ is
a L\'evy-driven OU type process described by the following SDE,
\begin{eqnarray}
\left\{\begin{array}{ll}
       dY(t)=-\Lambda Y(t^{-})dt+dL(\lambda t),\\
       Y(0)=y_{0},
       \end{array}
\right. \elabel{sdeou}
\end{eqnarray}
where $\Lambda=\mbox{diag}(\lambda)$ and
$y_{0}=(y_{10},...,y_{h0})'$. Now, define
\begin{eqnarray}
b(y)&\equiv&(b_{1}(y),...,b_{d}(y))':R_{c}^{h}
\rightarrow[0,\infty)^{d},
\nonumber\\
\sigma(y)&\equiv&(\sigma_{mn}(y))_{d\times d}:\;\;\;\;\;\;
R_{c}^{h}\rightarrow(0,\infty)^{dd},\nonumber
\end{eqnarray}
where $R_{c}^{h}\equiv(c_{1},\infty)\times...\times(c_{h},\infty)$
with $c_{i}=y_{i0}e^{-\lambda_{i}T}$. Thus, we can impose the
following conditions related to the coefficients in
\eq{stockassetm}-\eq{sdeou}:

{\bf C1.} The functions $b(y)$ and $\sigma(y)$ are continuous in $y$
and satisfy that, for each $y\in R_{c}^{h}$,
\begin{eqnarray}
\|b(y)\|&\leq&
A_{b}+B_{b}\|y\|,\elabel{linearg}\\
\|\sigma(y)\sigma(y)'\|&\leq&
A_{\sigma}+B_{\sigma}\|y\|,\elabel{lineargI}\\
\left\|\left(\sigma(y)\sigma(y)'\right)^{-1}\right\|&\leq&
\frac{1}{b_{\sigma}\|y\|}, \elabel{lineargII}
\end{eqnarray}
where the norm $\|A\|$ takes the largest absolute value of all
components of a vector $A$ or all entries of a matrix $A$, and
$A_{b}\geq 0,A_{\sigma}\geq 0,B_{b}\geq 0,B_{\sigma}\geq 0$,
$b_{\sigma}>0$ are constants.

{\bf C2.} The derivatives $\frac{\partial b(y)}{\partial y_{i}}$ and
$\frac{\partial (\sigma(y)\sigma(y)')^{-1}} {\partial y_{i}}$ for
all $i\in\{1,...,h\}$ are continuous in $y$ and satisfy that, for
each $y\in R_{c}^{h}$,
\begin{eqnarray}
\left\|\frac{\partial b(y)}{\partial y_{i}}\right\|
&\leq&\bar{A}_{b}+\bar{B}_{b}\|y\|,\elabel{derivcon}\\
\left\|\frac{\partial (\sigma(y)\sigma(y)')^{-1}} {\partial
y_{i}}\right\|&\leq&\bar{A}_{\sigma}+\bar{B}_{\sigma}\|y\|,
\elabel{derivconI}
\end{eqnarray}
where $\bar{A}_{b}$, $\bar{A}_{\sigma}$, $\bar{B}_{b}$ and
$\bar{B}_{\sigma}$ are some nonnegative constants.
%\begin{remark}
%Our market model under conditions {\bf C1} and {\bf C2} includes the
%well-known financial model with non-Gaussian stochastic volatility
%of Ornstein-Uhlenbeck type (see, e.g., Barndorf-Nielsen $\&$
%Shephard~\cite{barnie:nonorn} and Benth {et
%al.}~\cite{benkar:merpor}) as a special case. In addition, we can
%construct more examples that satisfy conditions {\bf C1} and {\bf
%C2}.
%\end{remark}

We now introduce the conditions for each subordinator $L_{i}$ with
$i\in\{1,...,h\}$, which can be represented by (e.g., Theorem 13.4
and Corollary 13.7 in Kallenberg~\cite{kal:foumod})
\begin{eqnarray}
L_{i}(t)=\int_{(0,t]}\int_{z_{i}>0}z_{i}N_{i}(ds,dz_{i}), \;t\geq 0.
\elabel{subordrep}
\end{eqnarray}
Here and in the sequel, $N_{i}((0,t]\times A)\equiv\sum_{0<s\leq t}
I_{A}(L_{i}(s)-L_{i}(s^{-}))$ denotes a Poisson random measure with
deterministic, time-homogeneous intensity measure
$\nu_{i}(dz_{i})ds$. $I_{A}(\cdot)$ is the index function over the
set $A$. $\nu_{i}$ is the L\'{e}vy measure satisfying
\begin{eqnarray}
\int_{z_{i}>0}\left(e^{Cz_{i}}-1\right)\nu_{i}(dz_{i})<\infty
\elabel{expintcon}
\end{eqnarray}
with $C$ taken to be a sufficiently large positive constant to
guarantee all of the related integrals in this paper meaningful.
Note that the condition in \eq{expintcon} is on the integrability of
the tails of the L\'evy measures (readers are referred to
Dai~(\cite{dai:meavar,dai:contru,dai:opthed,dai:heatra,dai:optrat})
for the justification of its reasonability).

\subsection{Admissible Strategies}

First, we use $D(t)=(D_{1}(t),...,D_{d}(t))'$ to denote the
associated $d$-dimensional discounted price process, i.e., for each
$m\in\{1,...,d\}$,
\begin{eqnarray}
D_{m}(t)=\frac{S_{m}(t)}{S_{0}(t)}=e^{-rt}S_{m}(t).
\elabel{discountp}
\end{eqnarray}
Furthermore, we define $L^{2}_{{\cal F}}\left([0,T],R^{d},P\right)$
to be the set of all $R^{d}$-valued measurable stochastic processes
$Z(t)$ adapted to $\{{\cal F}_{t},t\in[0,T]\}$ such that
$E\left[\int_{0}^{T}\|Z(t)\|^{2}dt\right]<\infty$. Thus, it follows
from Lemma~\ref{disprice} that $D(\cdot)$ is a continuous $\{{\cal
F}_{t}\}$-semimartingale. In addition, $D(\cdot)$ is locally in
$L^{2}_{{\cal F}}([0,T],$ $R^{d},P)$, i.e., there is a localizing
sequence of stopping times $\{\sigma_{n}\}$ with $n\in {\cal
N}\equiv\{0,1,2,...\}$ such that, for any $n\in{\cal N}$,
\begin{eqnarray}
&&\sup\{E\left[D^{2}(\tau)\right]:\mbox{all
stopping}\;\;\tau\;\;\mbox{time
satisfying}\;\tau\leq\sigma_{n}\}<\infty. \elabel{localmean}
\end{eqnarray}

Second, let $L(D)$ denote the set of $D$-integrable and predictable
processes in the sense of Definition 6.17 in page 207 of Jacod and
Shiryaev~\cite{jacshi:limthe}. Furthermore, let $u_{i}(t)$ denote
the number of shares invested in stock $i\in\{1,...,d\}$ at time $t$
and define $u(t)\equiv(u_{1}(t),...,u_{d}(t))'$. Then, we have the
following definitions concerning admissible strategies.
%It follows from Theorem 4.5(a) in page 180 of Jacod and
%Shiryaev~\cite{jacshi:limthe} that, for each $u\in L(D)$, we have,
%\begin{eqnarray}
%&&(u\cdot D)(t)=\lim_{k\rightarrow\infty}
%\sum_{i=1}^{d}\int_{0}^{t}u_{i}(s)I_{\{\|u(s)\|\leq
%k\}}dM^{D}_{i}(s) +\sum_{i=1}^{d}\int_{0}^{t}u_{i}(s)dB^{D}_{i}(s),
%\elabel{stointud}
%\end{eqnarray}
%where the limit in the first term on the right-hand side of
%\eq{stointud} corresponds to the convergence in probability
%uniformly on every compact set of $[0,T]$.
\begin{definite}
An $R^{d}$-valued trading strategy $u$ is called simple if it is a
linear combination of strategies $ZI_{(\tau_{1},\tau_{2}]}$ where
$\tau_{1}\leq\tau_{2}$ are stopping times dominated by $\sigma_{n}$
for some $n\in{\cal N}$ and $Z$ is a bounded ${\cal
F}_{\tau_{1}}$-measurable random variable. Furthermore, the set of
all such simple trading strategies is denoted by ${\Theta(D)}$.
\end{definite}
\begin{definite}\label{mvset}
A trading strategy $u\in L(D)$ is called admissible if there is a
sequence $\{u^{n},n\in{\cal N}\}$ of simple strategies such that:
$\left(u^{n}\cdot D\right)(t)\rightarrow (u\cdot D)(t)$ in
probability as $n\rightarrow\infty$ for any $t\in[0,T]$ and
$\left(u^{n}\cdot D\right)(T)\rightarrow (u\cdot D)(T)$ in
$L^{2}(P)$ as $n\rightarrow\infty$. Furthermore, the set of all such
admissible strategies is denoted by $\bar{\Theta}(D)$.
\end{definite}
%Let $\bar{A}$ denote the closure of a set $A\in L^{2}(P)$, then it
%follows from Corollary 2.9 in C\u{e}rn\'y and
%Kallsen~\cite{cerkal:strgen} that
%\begin{eqnarray}
%K_{2}\equiv\overline{\{(u\cdot D)(T):u\in\Theta(D)\}}=\{(u\cdot
%D)(T):u\in\bar{\Theta}(D)\} \elabel{setrelation}.
%\end{eqnarray}

\section{Main Theorem}\label{optimalhedging}

First, for each $y\in R_{c}^{h}$, define
\begin{eqnarray}
B(y)&\equiv&(b_{1}(y)-r,...,b_{d}(y)-r)',\elabel{rhobs}\\
\rho(y)&\equiv&B(y)'\left[\sigma(y)\sigma(y)'\right]^{-1}B(y),
\elabel{rhobsI}\\
P(t,y)&\equiv&E_{t,y}\left[e^{-\int_{t}^{T}\rho(Y(s))ds} \right]>0,
\elabel{pexactsolution} \\
O(t)&\equiv&P(t,Y(t)),\elabel{opyt}\\
a(t)&\equiv&(\mbox{diag}(D(t)))^{-1}
\left(\sigma(Y(t^{-}))\sigma(Y(t^{-}))'\right)^{-1}
B(t,Y(t^{-})),
\elabel{adjustmentp}\\
\hat{Z}(t)&\equiv&\frac{O(t){\cal E}(-a\cdot
D)(t)}{O_{0}},\;\;O_{0}=O(0).\elabel{densitym}
\end{eqnarray}
Note that the process $a(\cdot)$ presented in \eq{adjustmentp} is
corresponding to the adjustment process defined in Lemma 3.7 of
Cerny and Kallsen~\cite{cerkal:strgen}. Furthermore, the process
$\hat{Z}(\cdot)$ presented in \eq{densitym} is associated with the
density process defined in Proposition 3.13 of Cerny and
Kallsen~\cite{cerkal:strgen}. In addition, here and in the sequel,
${\cal E}(N)=\{{\cal E}(N)(t),t\in[0,T]\}$ denotes the stochastic
exponential for a univariant continuous semimartingale
$N=\{N(t),t\in[0,T]\}$ (e.g., pages 84-85 of
Protter~\cite{pro:stoint}) with
\begin{eqnarray}
{\cal
E}(N)(t)=\exp\left\{N(t)-\frac{1}{2}[N,N](t)\right\}\elabel{exps}
\end{eqnarray}
where $[\cdot,\cdot]$ denotes the quadratic variation process of
$N$.

Second, let $L^{2}_{{\cal F},p}([0,T],R^{d},P)$ denote the set of
all $R^{d}$-valued predictable processes (see, e.g., Definition 5.2
in page 21 of Ikeda and Watanabe~\cite{ikewat:stodif}) and let
$L^{2}_{p}([0,T],$ $R^{h},P)$ be the set of all $R^{h}$-valued
predictable processes $\tilde{Z}(t,z)=$ $(\tilde{Z}_{1}(t,z),$
$...,$ $\tilde{Z}_{h}(t,z))'$ satisfying
\begin{eqnarray}
&&E\left[\sum_{i=1}^{h}\int_{0}^{T}\int_{z_{i}>0}
\left|\tilde{Z}_{i}(t,z)\right|^{2}\nu_{i}(dz_{i})dt\right]<\infty.
\nonumber
\end{eqnarray}
Furthermore, let
\begin{eqnarray}
\bar{Z}(t)&\equiv&\frac{\hat{Z}(t^{-})}{\hat{Z}(t)},\elabel{barbii}\\
\bar{B}_{i}(Y(t^{-}))&\equiv&\sum_{j=1}^{d}\left(\left(B(Y(t^{-}))'
\left(\sigma(Y(t^{-}))\sigma(Y(t^{-}))'\right)^{-1}\right)\right)_{j}
\sigma_{ji}(Y(t^{-})),
\elabel{barbiiI}\\
F(t,z_{i}))&\equiv&\frac{P(t,Y(t^{-})+z_{i}e_{i})
-P(t,Y(t^{-}))}{P(t,Y(t^{-}))}, \elabel{fszy}
\end{eqnarray}
where, $e_{i}$ is the $h$-dimensional unit vector with the $i$th
component one. Then, we define
\begin{eqnarray}
&&g\left(t,V(t^{-}),\bar{V}(t),\tilde{V}(t,\cdot),Y(t^{-})\right)
\elabel{mgvvv}\\
&\equiv&-\sum_{i=1}^{d}\bar{V}_{i}(t)\bar{B}_{i}(Y(t^{-}))\nonumber\\
&&+\sum_{i=1}^{h}\int_{z_{i}>0}\left(\tilde{V}_{i}(t,z_{i})F(t,z_{i})
\bar{Z}(t)+V(t^{-})\left(F(t,z_{i})\bar{Z}(t)\right)^{2}\right)
\lambda_{i}\nu_{i}(dz_{i}).\nonumber
\end{eqnarray}
\begin{definite}
For a given random variable $H$, a 3-tuple $(V,\bar{V},\tilde{V})$
is called a $\{{\cal F}_{t}\}$-adapted strong solution of the BSDE
\begin{eqnarray}
\;\;\;\;\;V(t)
&=&H-\int_{t}^{T}g(s,V(s^{-}),\bar{V}(s),\tilde{V}(s,\cdot),Y(s^{-}))ds
\elabel{xbsde}\\
&&-\int_{t}^{T}\sum_{i=1}^{d}\bar{V}_{i}(s)dW_{i}(s)
-\int_{t}^{T}\sum_{i=1}^{h}\int_{z_{i}>0}
\tilde{V}_{i}(s,z_{i})\tilde{N}_{i}(\lambda_{i}ds,dz_{i})\nonumber
\nonumber
\end{eqnarray}
if $V\in L^{2}_{{\cal F}}([0,T],R,P)$ is a c\`adl\`ag process,
$\bar{V}=(\bar{V}_{1},...,\bar{V}_{d})\in L_{{\cal
F},p}^{2}([0,T],R^{d},P)$,
$\tilde{V}=(\tilde{V}_{1},...,\tilde{V}_{h})\in
L^{2}_{p}([0,T],R^{h},P)$, and \eq{xbsde} holds a.s., where
\begin{eqnarray}
\tilde{N}_{i}(\lambda_{i}dt,dz_{i})\equiv
N_{i}(\lambda_{i}dz_{i},dt)-\lambda_{i}\nu_{i}(dz_{i})dt\;\;\;
\mbox{for each}\;\;i\in\{1,...,h\}. \elabel{nilambda}
\end{eqnarray}
\end{definite}

To impose suitable condition on the option $H$, we use
$L^{\gamma}_{{\cal F}_{T}}(\Omega,R^{d},P)$ for a positive integer
$\gamma$ to denote the set of all $R^{d}$-valued, ${\cal
F}_{T}$-measurable random variables $\xi\in R^{d}$ satisfying
$E\left[\|\xi\|^{\gamma}\right]<\infty$.
\begin{assumption}\label{hstopasump}
$H\in L^{4}_{{\cal F}_{T}}(\Omega,R,P)$ and there exists a sequence
of random variables $H_{\tau_{n}}\in L^{2}_{{\cal
F}_{T\wedge\tau_{n}}}(\Omega,R,P)$ satisfying
$H_{\tau_{n}}\rightarrow H$ in $L^{2}$ as $n\rightarrow\infty$ and
$H_{\tau_{n}}(\omega)=H(\omega)$ for all
$\omega\in\{\omega,\tau_{n}(\omega)\geq T\}$, where $\{\tau_{n}\}$
is a sequence of nondecreasing $\{{\cal F}_{t}\}$-stopping times
satisfying $\tau_{n}\rightarrow\infty$ a.s. as $n\rightarrow\infty$.
\end{assumption}
As pointed out in Dai~\cite{dai:opthed}, under conditions {\bf C1},
{\bf C2}, and \eq{expintcon}, the discounted European call and put
options satisfy Assumption~\ref{hstopasump}. Now, we can state our
main theorem of the paper as follows.
\begin{theorem}\label{opthedge}
Under conditions {\bf C1}, {\bf C2}, \eq{expintcon}, and
Assumption~\ref{hstopasump}, let $(V,\bar{V},\tilde{V})$ be the
unique $\{{\cal F}_{t}\}$-adapted strong solution of the BSDE in
\eq{xbsde}. Then, the optimal hedging strategy
$\phi\in\bar{\Theta}(D)$ for \eq{mvhedgep} is given by
\begin{eqnarray}
\phi(t)=\xi(t)-(v+\Psi(t^{-})-V(t^{-}))a(t),
\elabel{optimalh}
\end{eqnarray}
where, the pure hedge coefficient $\xi$ is given by
\begin{eqnarray}
\xi(t)&=&\left(\tilde{c}^{D^{*}}(t)\right)^{-1}
\left(\tilde{c}^{DV^{*}}(t)\right),\elabel{xicb}\\
\tilde{c}^{D^{*}}(t)&=&\mbox{diag}(D(t))
\left(\sigma(Y(t^{-}))\sigma(Y(t^{-}))'\right)\mbox{diag}(D(t)),
\elabel{xicbI}\\
\tilde{c}^{DV^{*}}(t)&=&
\left(\sum_{i=1}^{d}D_{1}(t)\sigma_{1i}(Y(t^{-}))\bar{V}_{i}(t),...,
\sum_{i=1}^{d}D_{d}(t)\sigma_{di}(Y(t^{-}))\bar{V}_{i}(t)\right)'.
\elabel{xicbII}
\end{eqnarray}
In addition, $\Psi$ is the unique solution of the SDE
\begin{eqnarray}
\Psi(t)=((\xi-(v-V_{-})a)\cdot D)(t) -(\Psi_{-}\cdot(a\cdot D))(t).
\elabel{geqn}
\end{eqnarray}
\end{theorem}
\begin{remark}\label{rrkkm}
The process $V(\cdot)$ appeared in Theorem~\ref{opthedge} is
actually the conditional mean value process,
\begin{eqnarray}
V(t)=E_{Q^{*}}\left[H\left|{\cal
F}_{t}\right]\right.\;\;\mbox{with}\;\;dQ^{*}\equiv\hat{Z}(T)dP.
\elabel{defineV}
\end{eqnarray}
Since it is not easy to be computed directly as the Markovian based
conditional process $O(t,Y(t))$, we turn to use the BSDE in
\eq{xbsde} to evaluate it, which is convenient for us to design the
optimal hedging policy as explained in Introduction of the paper.
\end{remark}

The proof of Theorem~\ref{opthedge} will be provided in
Section~\ref{hedgee}.
%\begin{remark}
%The BSDE \eq{xbsde} can be solved by developing a numerical
%algorithm (see, e.g., Ma {et al.}~\cite{mapro:nummet}), and once we
%establish Theorem~\ref{opthedge}, one can easily to get the
%associated hedging error (readers are referred to Theorem 4.12 in
%C\u{e}rn\'y and Kallsen~\cite{cerkal:strgen} for more details).
%\end{remark}

\section{Performance Comparisons}\label{necom}

The material in this section is partially reported in the short
conference version of the current paper (see,
Dai~\cite{dai:opthed}). To be convenient and clear for readers, we
refine it here. Note that the interest rate $r$ in \eq{bankaset}
here is taken to be zero. Furthermore, the financial market is
assumed to be self-financing, which implies that $X(t)=v+(u\cdot
D)(t)$. In addition, the terminal option $H$ is taken to be a
constant $p$, i.e., $H=p$. In this case, the optimal policies can be
explicitly obtained by the feedback control method studied
in~Dai~\cite{dai:meavar} and the martingale method presented in the
current paper. In the late method, the related BSDE is a degenerate
one, which can be easily observed from \eq{defineV} in
Remark~\ref{rrkkm}. However, from this constant option $H=p$, we can
construct two insightful examples to provide the effective
comparisons between the two methods.

More precisely, by (18) in Theorem 3.1 of Dai~\cite{dai:meavar}, we
know that the terminal variance under the optimal policy stated in
(15) of Theorem 3.1 of Dai~\cite{dai:meavar} is given by
\begin{eqnarray}
Var(X^{*}(T))=\frac{P(0,y_{0})} {1-P(0,y_{0})}\left(p-v\right)^{2}.
\elabel{onevar}
\end{eqnarray}
In addition, by using Theorem~\ref{opthedge} in the current paper
and Theorem 4.12 in C\u{e}rn\'y and Kallsen~\cite{cerkal:strgen}, we
know that the hedging error under the optimal policy in
\eq{optimalh} is given by
\begin{eqnarray}
Herr=P(0,y_{0})\left(p-v\right)^{2}. \elabel{twovar}
\end{eqnarray}
For the purpose of performance comparisons, we calculate the
differences between the optimal terminal variances in \eq{onevar}
and the optimal hedging errors in \eq{twovar}, i.e.,
\begin{eqnarray}
Error&=&Var(X^{*}(T))-Herr\elabel{errordif}\\
&=&\frac{(P(0,y_{0}))^{2}}
{1-P(0,y_{0})}\left(p-v\right)^{2}\nonumber\\
&>&0.\nonumber
\end{eqnarray}
The result shown in the last inequality of \eq{errordif} is
intuitively right since the optimal strategy in \eq{optimalh} is
taken over a general decision set given in Definition~\ref{mvset}
and the one in (15) of Theorem 3.1 of Dai~\cite{dai:meavar} is taken
in an {\em ad-hoc} approach. Nevertheless, the errors are very small
as displayed in the following numerical examples.
\begin{example}\label{exI}
Here, we suppose that the financial market is given by the
Black-Scholes model
\begin{eqnarray}
dD(t)=D(t)(\alpha dt+\beta dB(t)),
\elabel{blachscholes}
\end{eqnarray}
where $\alpha$ and $\beta$ are given constants. Owing to Definition
2.1.4(b) in pages 273-274 of $\emptyset$ksendal \cite{oks:stodif},
the option $H=p$ (a positive constant) is not attainable and hence
the associated hedging error can not be zero if the initial
endowment $v\neq p$. However, by the simulated results displayed in
Figures~\ref{bsmodel} and~\ref{bsmodelI}, we see that the absolute
error between the optimal variance based on the policy in (15) of
Theorem 3.1 of Dai~\cite{dai:meavar} and the optimal hedging error
based on the strategy in \eq{optimalh} approaches zero as the
terminal time increases. The rate of convergence is heavily
dependent on the volatility $\beta$. If $\beta$ is relatively large,
the difference requires more time to reach zero. Nevertheless, if
the millisecond is employed to represent the time unit in a
supercomputer based trading system, the required time for the
convergence makes sense in practice.
\begin{figure}[tbh]
\centerline{\epsfxsize=4.0in\epsfbox{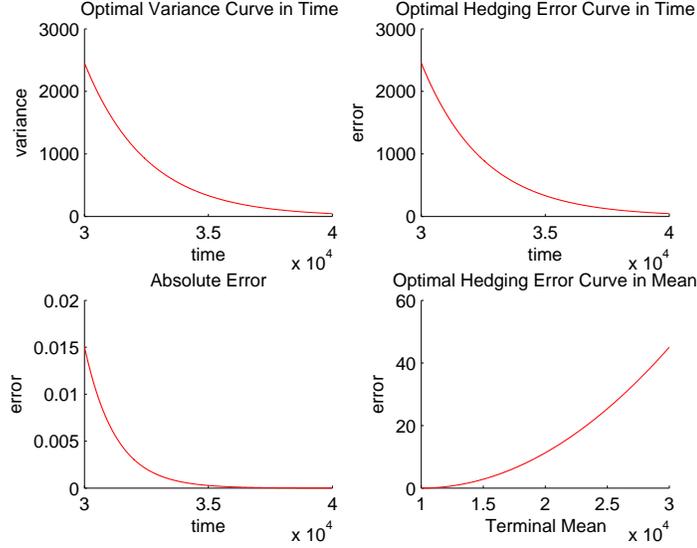}} \caption{\small
Errors using Black-Scholes model with $r=0$, $v=10000$, $p=30000$,
$T=40000$, $\alpha=2$, $\beta=100$.} \label{bsmodel}
\end{figure}
\begin{figure}[tbh]
\centerline{\epsfxsize=4.0in\epsfbox{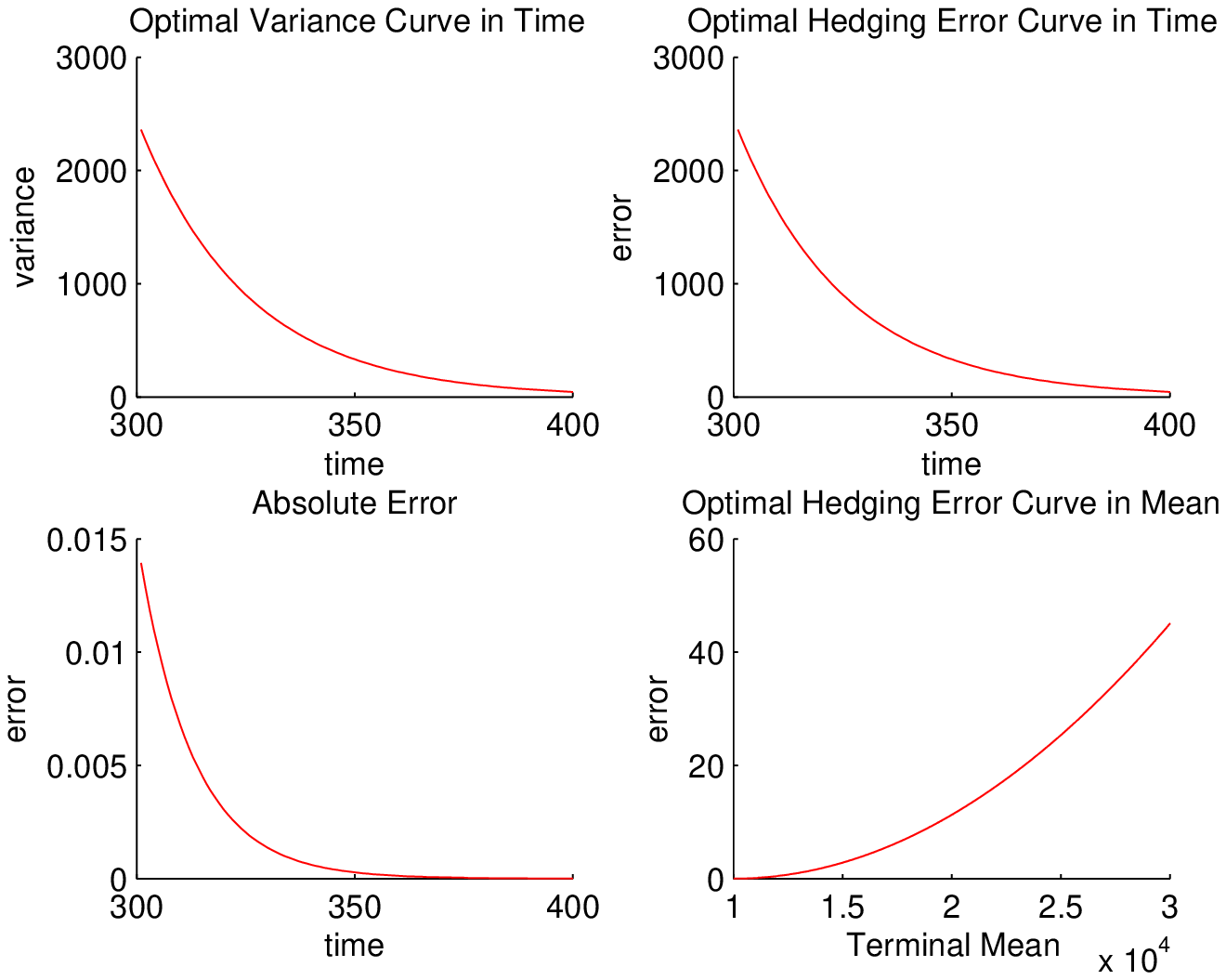}} \caption{\small
Errors using Black-Scholes model with $y_{0}=10$, $r=0$, $v=10000$,
$p=30000$, $T=400$, $\alpha=2$, $\beta=10$.} \label{bsmodelI}
\end{figure}
\end{example}
\begin{example}
Here, we assume that the financial market is presented by the BNS
model
\begin{eqnarray}
dD(t)=D(t)((\alpha+\beta Y(t^{-}))dt+\sqrt{Y(t^{-})}dB(t)),
\elabel{bnsm}
\end{eqnarray}
where $\alpha$ and $\beta$ are given constants. Furthermore, owing
to the remarks to the condition in \eq{expintcon} and owing to the
discussions in Dai~\cite{dai:contru}, we suppose that the driving
subordinator $L(\lambda\cdot)$ with $\lambda=1$ to the SDE in
\eq{sdeou} is a compound Poisson process. The interarrival times of
the process are exponentially distributed with mean $1/\mu$ and the
jump sizes of the process are also exponentially distributed with
mean $1/\mu_{1}$. By the simulated results displayed in
Figure~\ref{bnsmodel}, we see that the similar illustration
displayed in Example~\ref{exI} also makes sense for the current
example, where $\delta$ appeared in Figure~\ref{bnsmodel} is the
length of equally divided subintervals of $[0,T]$. In addition, by
the simulated results, we also see that, by perfect hedging is
impossible in an incomplete market, the mean-variance hedging errors
can be very small in many cases when terminal time increases.
%Finally, based on the solution to \eq{sdeou}, we present a sample
%path of $Y(t^{-})$ in Figure~\ref{bnsmodelI} to provide some insight
%about the stationarity and evolving of the volatility process.
\begin{figure}[tbh]
\centerline{\epsfxsize=4.0in\epsfbox{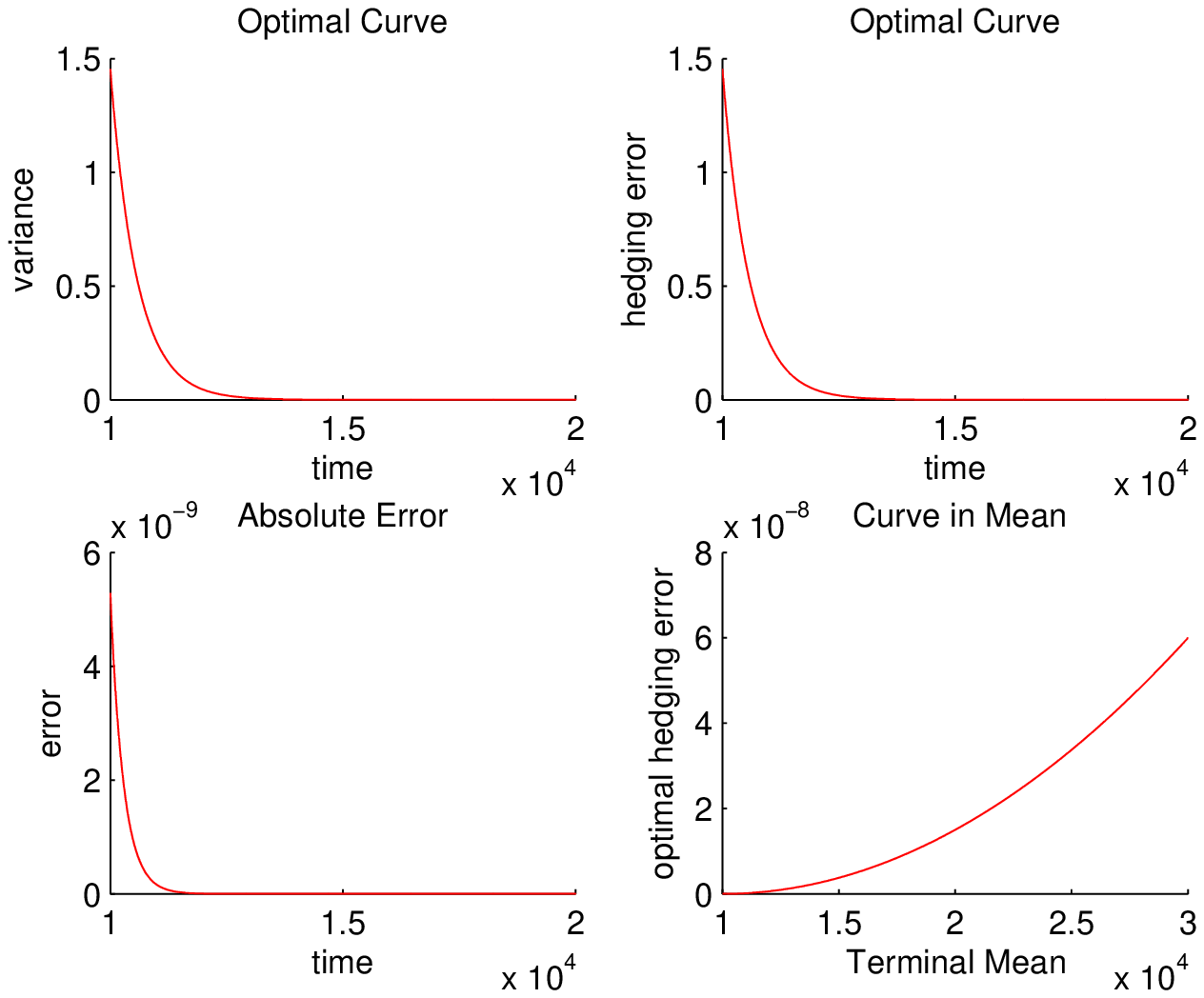}} \caption{\small
Errors using BNS model with $y_{0}=10$, $r=0$, $v=10000$, $p=30000$,
$T=200$, $\delta=0.01$, $\alpha=0.5$, $\beta=0.02$, $\mu=10$,
$\mu_{1}=8$.} \label{bnsmodel}
\end{figure}
%\begin{figure}[tbh]
%\centerline{\epsfxsize=3.0in\epsfbox{volatility.eps}}
%\caption{\small The volatility process for  BNS model with
%$y_{0}=10$, $r=0$, $v=1000$, $T=200$, $\delta=0.01$, $\mu=0.5$, and
%$\mu_{1}=8$.} \label{bnsmodelI}
%\end{figure}
\end{example}

\section{Proof of Theorem~\ref{opthedge}}\label{hedgee}

The proof consists of four parts presented in the subsequent four
subsections: the justification of a proposition related to the
discounted price process, the demonstration of a proposition related
to the VOMM, the illustration of unique existence of solution to a
type of BSDEs with jumps, and the remaining proof of
Theorem~\ref{opthedge}.

\subsection{The Proposition Related to the Discounted Price Process}\label{mtwolemma}

\begin{proposition}\label{disprice}
Under conditions {\bf C1}, {\bf C2}, and \eq{expintcon}, we have
that $D(\cdot)$ is a continuous $\{{\cal F}_{t}\}$-semimartingale,
i.e.,
\begin{eqnarray}
D(\cdot)=D_{0}+M^{D}(\cdot)+B^{D}(\cdot), \elabel{dsde}
\end{eqnarray}
where $M^{D}(\cdot)$ and $B^{D}(\cdot)$ are an $\{{\cal
F}_{t}\}$-martingale and a predictable process of finite variation
respectively. Furthermore, $D(\cdot)$ is locally in $L^{2}_{{\cal
F}}([0,T],$ $R^{d},P)$ in the sense as stated in \eq{localmean}.
\end{proposition}

We divide the proof of the proposition into two parts. First, we
have the following lemma.
\begin{lemma}\label{sprice}
Under \eq{expintcon}, the unique adapted solution to the SDE in
\eq{sdeou} for each $\hat{t}>t$, $i\in\{1,...,h\}$, and
$y\in(0,\infty)^{h}$ is given by
\begin{eqnarray}
&&Y_{i}(\hat{t})=y_{i}e^{-\lambda_{i}(\hat{t}-t)}
+\int_{t}^{\hat{t}}e^{-\lambda_{i}(s-t)} dL_{i}(\lambda_{i}s)\geq
y_{i}e^{-\lambda_{i}\hat{t}},\;\;\;\;\;Y_{i}(t)=y_{i}.
\elabel{uniqso}
\end{eqnarray}
Furthermore, under conditions {\bf C1}, {\bf C2}, and
\eq{expintcon}, there is a unique solution $(S_{0}(t),S(t)')$ for
\eq{stockassetm}-\eq{sdeou}, which is an $\{{\cal F}_{t}\}$-adapted
and continuous semimartingale with
\begin{eqnarray}
&&S(\cdot)\in L^{2}_{{\cal
F}}\left([0,T],R^{d},P\right).\elabel{sqareint}
\end{eqnarray}
In addition, for each $m\in\{1,...,d\}$,
\begin{eqnarray}
S_{m}(t)&=&S_{m}(0)\exp\left\{\int_{0}^{t}\left[b_{m}(Y(s^{-}))-
\frac{1}{2}\sum_{n=1}^{d}\sigma_{mn}^{2}(Y(s^{-}))\right]ds\right.
\elabel{exsol}\\
&&\left.+\int_{0}^{t}\sum_{n=1}^{d}\sigma_{mn}(Y(s^{-}))dW_{n}(s)\right\}.
\nonumber
\end{eqnarray}
\end{lemma}
\begin{proof}
The claim concerning \eq{uniqso} directly follows from pages 316-317
in Applebaum~\cite{app:levpro}. Furthermore, owing to conditions
{\bf C1} and {\bf C2}, we know that our market given by
\eq{stockassetm}-\eq{sdeou} satisfies the conditions as required by
Lemma 4.1 in Dai~\cite{dai:meavar}. Thus, our market has a unique
solution, which is $\{{\cal F}_{t}\}$-adapted, continuous, and
mean-square integrable as stated in Lemma~\ref{disprice}. In order
to prove \eq{exsol}, let
\begin{eqnarray}
&&X_{m}(t)=\int_{0}^{t}\alpha_{m}(Y(s^{-}))ds
+\int_{0}^{t}\beta_{m}(Y(s^{-}))'dW(s), \elabel{xalpha}
\end{eqnarray}
where, for any $s\in[0,T]$,
\begin{eqnarray}
\alpha_{m}(Y(s^{-}))&=&
b_{m}(Y(s^{-}))-\frac{1}{2}\sum_{n=1}^{d}\sigma_{mn}^{2}(Y(s^{-})),
\nonumber\\
\beta_{m}(Y(s^{-}))&=&(\sigma_{m1}(Y(s^{-})),...,\sigma_{md}(Y(s^{-})))'.
\nonumber
\end{eqnarray}
Then, by condition {\bf C1}, there exists some nonnegative constant
$D_{1}$ such that
\begin{eqnarray}
E\left[\int_{0}^{T}\left|\alpha_{m}(Y(s^{-}))\right|ds\right] &\leq&
D_{1}T+\left(B_{b}+\frac{1}{2}B_{\sigma}\right)T
e^{\sum_{i=1}^{h}y_{i0}}
\prod_{i=1}^{h}E\left[e^{L_{i}(\lambda_{i}T))}\right]
\elabel{rintegrable}\\
&<&\infty,
\nonumber
\end{eqnarray}
where we have used the facts that $L(\lambda t)$ is nonnegative and
nondecreasing in $t$, the independence assumption among
$L_{i}(\lambda_{i}\cdot)$ for $i\in\{1,...,h\}$, and
\begin{eqnarray}
&&a+b\|L(\lambda t)\|\leq\left(\frac{1}{\epsilon} \vee
a\right)e^{b\epsilon\|L(\lambda t)\|}\;\;\mbox{for any}\;\;
a\geq 0,\;b\geq 0,\;\epsilon>0, \elabel{usefuline}\\
&&Y^{t,y_{i}}_{i}(\hat{t})\leq y_{i}+L_{i}(\lambda_{i}\hat{t})-
L_{i}(\lambda_{i}t)\;\;\;\;\;\mbox{for any}\;\;\hat{t}\geq t,
\elabel{yleq}\\
&&E\left[e^{CL_{i}(\lambda_{i}t)}\right]
=\exp\left(\lambda_{i}t\int_{z_{i}>0}\left(e^{Cz_{i}}-1\right)
\nu_{i}(dz_{i})\right)<\infty. \elabel{basecon}
\end{eqnarray}
Similarly, we can show that
\begin{eqnarray}
E\left[\int_{0}^{T}\beta^{2}_{m}(Y(s^{-}))ds\right]<\infty.
\elabel{betami}
\end{eqnarray}
Note that $W(\cdot)$ and $L_{i}(\lambda_{i}\cdot)$ for
$i\in\{1,...,h\}$ are independent; $W$ is $\{{\cal
F}_{t},t\in[0,T]\}$-martingale; $\alpha_{m}(Y(t^{-}))$ and
$\beta_{m}(Y(t^{-}))$ are ${\cal F}_{t}$-adapted. Then, it follows
from Definition 4.1.1 in $\emptyset$ksendal \cite{oks:stodif} and
the associated It$\hat{o}$'s formula (e.g., Theorem 4.1.2 in
$\emptyset$ksendal \cite{oks:stodif}) that $S_{m}(t)$ given in
\eq{exsol} for each $m$ is the unique solution of \eq{stockassetm}.

Now, we show that $S_{m}(\cdot)$ for each $m\in\{1,...,d\}$ is a
square-integrable $\{{\cal F}_{t}\}$-semimartingale.
%Note that log($S_{m}(\cdot)$) is an $\{{\cal F}_{t}\}$-semimartingale. Then,
%it is a direct conclusion of Theorem 33 in page 221 of Dellacherie
%and Meyer~\cite{delmey:propot} that $S_{m}(\cdot)$ is an $\{{\cal
%F}_{t}\}$-semimartingale.
To do so, we rewrite \eq{stockassetm} in its integral form
\begin{eqnarray}
&&S_{m}(t)=S_{m}(0)+\int_{0}^{t}S_{m}(s)b_{m}(Y(s^{-}))ds
+\int_{0}^{t}S_{m}(s)\sum_{n=1}^{d}\sigma_{mn}(Y(s^{-}))dW_{n}(s).
\elabel{smint}
\end{eqnarray}
Then, the third term on the right-hand side of \eq{smint} is a
square-integrable $\{{\cal F}_{t}\}$-martingale. In fact, it follows
from \eq{uniqso} that, for each $i\in\{1,...,h\}$ and $\hat{t}>t$,
\begin{eqnarray}
\lambda_{i}\int_{t}^{\hat{t}}Y^{(t,y_{i})}_{i}(s)ds
&=&y_{i}+L_{i}(\lambda_{i}\hat{t})-L_{i}(\lambda_{i}t)
-Y^{(t,y_{i})}_{i}(\hat{t})
\elabel{disequine}\\
&\leq&y_{i}+L_{i}(\lambda_{i}\hat{t})-L_{i}(\lambda_{i}t)
\nonumber\\
&=&y_{i}+L_{i}(\lambda_{i}(\hat{t}-t)),
\nonumber
\end{eqnarray}
where the last equality in \eq{disequine} holds in distribution.
Thus, it follows from Condition {\bf C1} and \eq{exsol} in
Lemma~\ref{disprice} that
\begin{eqnarray}
E\left[\int_{0}^{T}\left(S_{m}(s)
\sum_{n=1}^{d}\sigma_{mn}(Y(s^{-}))\right)^{2}ds\right]
&\leq&ds_{m}^{2}CT^{\frac{1}{2}}\left(E\left[e^{C\|L(\lambda
T)\|}\right]\right)^{\frac{1}{2}}\elabel{dmeight}\\
&<&\infty, \nonumber
\end{eqnarray}
where $C$ is some positive constant and we have used Theorem 39 in
page 138 of Protter~\cite{pro:stoint} and the condition
\eq{expintcon}. Therefore, by Theorem 4.40(b) in page 48 of Jacod
and Shiryaev~\cite{jacshi:limthe}, we know that the third term in
\eq{smint} is a square-integrable $\{{\cal F}_{t}\}$-martingale.

Furthermore, by the same method, we can show that the second term on
the right-hand side of \eq{smint} is of finite variation a.s. and is
square-integrable over $[0,T]$. Therefore, we conclude that
$S_{m}(\cdot)$ for each $m\in\{1,...,d\}$ is a square-integrable
$\{{\cal F}_{t}\}$-semimartingale. Hence, we complete the proof of
Lemma~\ref{sprice}. $\Box$
\end{proof}

\vskip 0.2cm \noindent{\bf Proof of Proposition~\ref{disprice}}

\noindent It follows from Lemma~\ref{sprice} and the Ito's formula
that, for each $m\in\{1,...,d\}$,
\begin{eqnarray}
&&B^{D}_{m}(t)=\int_{0}^{t}D_{m}(s)(b_{m}(Y(s^{-}))-r)ds,
\elabel{dsdeb}\\
&&M^{D}_{m}(t)=\int_{0}^{t}D_{m}(s)
\sum_{n=1}^{d}\sigma_{mn}(Y(s^{-}))dW_{n}(s). \elabel{mdreps}
\end{eqnarray}
Note that, by the similar calculation as in \eq{dmeight}, we have
\begin{eqnarray}
&&E\left[\int_{0}^{t}\left(D_{m}(s)\sum_{n=1}^{d}\sigma_{mn}(Y(s^{-}))
\right)^{2}ds\right]<\infty\elabel{mdmartingale}
\end{eqnarray}
for all $t\in[0,T]$. Thus, it follows from Theorem 4.40(b) in page
48 of Jacod and Shiryaev~\cite{jacshi:limthe} that $M^{D}$ is an
$\{{\cal F}_{t}\}$-martingale. Furthermore, it follows from a
similar explanation with the end of the proof for Lemma~\ref{sprice}
that $B^{D}$ is a predictable process of finite variation and
square-integrable. Thus, we know that $D$ is a continuous $\{{\cal
F}_{t}\}$-semimartingale. Moreover, it is locally in $L^{2}(P)$
since we may take $\sigma_{n}\equiv\inf\{\tau:D^{2}(\tau)\geq n\}$
as the sequence of localizing times. Hence, we complete the proof of
Proposition~\ref{disprice}. $\Box$

\subsection{A Proposition Related to the VOMM}

First of all, we use ${\cal P}_{D}(\bar{\Theta})(D)$ to denote the
set of all signed $\bar{\Theta}$-martingale measures in the sense
that $Q(\Omega)=1$ and $Q\ll P$ with
\begin{eqnarray}
&&\frac{dQ}{dP}\in
L^{2}(P)\;\;\mbox{and}\;\;E\left[\frac{dQ}{dP}(u\cdot
D)(T)\right]=0\nonumber
\end{eqnarray}
for a signed measure $Q$ on $(\Omega,{\cal F})$ and all
$u\in\bar{\Theta}(D)$. Then, we have the following proposition.
\begin{proposition}\label{equimar}
Under conditions {\bf C1}, {\bf C2}, and \eq{expintcon}, the
following claims are true:
\begin{enumerate}
\item $\hat{Z}$ is a $\{{\cal F}_{t}\}$-martingale, where
$\hat{Z}(\cdot)$ is given in \eq{densitym};
\item The measure $Q^{*}$ defined in \eq{defineV}
is an equivalent martingale measure (EMM), and $Q^{*}\in {\cal
U}_{2}^{e}(D)$ that is defined in \eq{equivmeasure};
\item The measure $Q^{*}$ is the VOMM in the
sense that
\begin{eqnarray}
Var\left(\frac{dQ^{*}}{dP}\right)=\min_{Q\in {\cal
P}_{D}(\bar{\Theta})}Var\left(\frac{dQ}{dP}\right).
%=\min_{Q\in {\cal
%P}_{D}(\bar{\Theta})}E\left[\left(\frac{dQ}{dP}-1\right)^{2}\right].
\nonumber
\end{eqnarray}
\end{enumerate}
\end{proposition}

We divide the proof of the proposition into demonstrating six lemmas
as follows.
\begin{lemma}\label{pintdif}
Under conditions {\bf C1}, {\bf C2}, and \eq{expintcon}, $P(t,y)$
defined in \eq{pexactsolution} is a solution of the following IPDE
\begin{eqnarray}
&&\left\{\begin{array}{ll} \frac{\partial}{\partial t}P(t,y)
=\rho(y)P(t,y) +\sum_{i=1}^{h}\lambda_{i}y_{i}
\frac{\partial}{\partial y_{i}}P(t,y)\\
\;\;\;\;\;\;\;\;\;\;\;\;\;\;\;\;\;\;\; -\sum_{i=1}^{h}\lambda_{i}
\int_{z_{i}>0}(P(t,y+z_{i}e_{i})-P(t,y))\nu_{i}(dz_{i}),\\
P(T,y)=1.
\end{array}
\right. \elabel{rhointdifeq}
\end{eqnarray}
for $y\in R_{c}^{h}$. Furthermore, we have
\begin{eqnarray}
&&P(t,y)\in C^{1,1}([0,T)\times R_{c}^{h},R^{1}),
\elabel{ppcontinuity}\\
&& E\left[\int_{0}^{T}|P(t,Y(t^{-}))|^{2}dt\right] <\infty,
\elabel{ppsisquint}\\
&&\sum_{i=1}^{h}E\left[\int_{0}^{T}\int_{z_{i}>0}
|P(t,Y(t^{-})+z_{i}e_{i})-P(t,Y(t^{-}))|^{2}
\nu(dz_{i})dt\right]<\infty. \elabel{psquarecon}
\end{eqnarray}
\end{lemma}
\begin{proof}
It follows from conditions {\bf C1}, {\bf C2}, and \eq{uniqso} that,
for each $i\in\{1,...,h\}$,
\begin{eqnarray}
&&\|\rho(Y(t))\|\leq A_{\rho}+B_{\rho}\|Y(t)\|,
\elabel{boundrho}\\
&&\left\|\frac{\partial\rho(Y(t))}{\partial y_{i}}\right\|\leq
\bar{A}_{1}+\bar{A}_{2}\|Y(t)\|+\bar{A}_{3}\|Y(t)\|^{2}
+\bar{A}_{4}\|Y(t)\|^{3}, \elabel{boundrhoI}
\end{eqnarray}
where $\bar{A}_{i}$ for $i\in\{1,2,3,4\}$ are some nonnegative
constants, $A_{\rho}$ and $B_{\rho}$ are given by
\begin{eqnarray}
A_{\rho}=\frac{2(A_{b}+r)B_{b}}{b_{\sigma}}
+\frac{(A_{b}+r)^{2}}{b_{\sigma}K},\;\;\;\;\;\;
B_{\rho}=\frac{B_{\sigma}^{2}}{b_{\sigma}}\nonumber
\end{eqnarray}
with $K=\min\{y_{i0}e^{-\lambda_{i}T},i=1,...,h\}$. Then, based on
an idea as used in Benth {\em at al.}~\cite{benkar:merpor}, we can
prove Lemma~\ref{pintdif} by the following four steps.

First, by direct calculation, we know that $P(t,y)$ is finite for
any $(t,y) \in[0,T]\times R_{c}^{h}$, i.e.,
\begin{eqnarray}
P(t,y)\leq\exp\left(K_{1}(T-t)+B_{\rho}\sum_{i=1}^{h}
\frac{y_{i}}{\lambda_{i}}\right)<\infty, \elabel{expbound}
\end{eqnarray}
where the nonnegative constant $K_{1}$ is given by
\begin{eqnarray}
K_{1}=A_{\rho}+\sum_{i=1}^{h}\lambda_{i}\int_{z_{i}>0}
\left(e^{\frac{B_{\rho}z_{i}}{\lambda_{i}}}-1\right)\nu_{i}(dz_{i}).
\nonumber
\end{eqnarray}
%As a matter of fact, by \eq{disequine} and \eq{basecon}, we have
%\begin{eqnarray}
%P(t,y)&\leq&
%E\left[e^{\int_{t}^{T}(A_{\rho}+B_{\rho}\|Y^{t,y}(s)\|)ds}\right]
%\elabel{pexpbound}\\
%&\leq& e^{A_{\rho}(T-t)}E_{t,y}\left[e^{\sum_{i=1}^{h}\left(
%B_{\rho}\int_{t}^{T}Y_{i}(s)ds\right)}\right]
%\nonumber\\
%&=& e^{A_{\rho}(T-t)}E_{t,y}\left[\prod_{i=1}^{h}e^{B_{\rho}
%\int_{t}^{T}Y_{i}(s)ds}\right]
%\nonumber\\
%&\leq& e^{A_{\rho}(T-t)}\prod_{i=1}^{h}
%\exp\left(\frac{B_{\rho}y_{i}}{\lambda_{i}}
%+\lambda_{i}(T-t)\int_{z_{i}>0}
%\left(e^{\frac{B_{\rho}z_{i}}{\lambda_{i}}}-1\right)
%\nu_{i}(dz_{i})\right)
%\nonumber\\
%&=&\exp\left(K_{1}(T-t)+B_{\rho}\sum_{i=1}^{h}
%\frac{y_{i}}{\lambda_{i}}\right) \nonumber\\
%&<&\infty\nonumber
%\end{eqnarray}
%where in the fourth inequality of \eq{pexpbound}, we used the fact
%that $B_{\rho}/\lambda_{i}\leq C$ and the independent assumption on
%$Y_{i}$ with $i\in\{1,...,h\}$.

Second, we prove that $P\in C^{0,1}\left([0,T]\times
R_{c}^{h},R^{1}\right)$ and the mapping
$(t,y)\rightarrow\frac{\partial P}{\partial y_{i}}(t,y)$ for each
$i\in\{1,...,h\}$ is continuous.. The continuity of $P(\cdot,y)$ for
each $y\in R_{c}^{h}$ can be shown as follows. Owing to the
condition \eq{linearg} and the fact \eq{disequine}, we know that
\begin{eqnarray}
\exp\left(\int_{t}^{T}\rho(Y^{t,y}(s))ds\right)
%&\leq&
%\exp\left(\int_{t}^{T}\left(A_{\rho}+B_{\rho}\|Y^{t,y}(s)\|\right)
%ds\right)
%\elabel{gcontiny}\\
%&\leq& \exp\left(A_{\rho}T+\sum_{i=1}^{h}\left(
%B_{\rho}\int_{t}^{T}Y^{t,y_{i}}_{i}(s)ds\right)\right)
%\nonumber\\
\leq\exp\left(A_{\rho}T+\sum_{i=1}^{h}\frac{B_{\rho}}{\lambda_{i}}
(y_{i}+L_{i}(\lambda_{i}T))\right). \elabel{gcontiny}
\end{eqnarray}
By \eq{expintcon} and \eq{basecon}, we know that the function on the
right-hand side of \eq{gcontiny} is integrable for each fixed $y\in
R_{c}^{h}$. Then, it follows from the Lebesgue's dominated
convergence theorem that $P(t,y)$ for each $y$ is continuous in
terms of $t\in[0,T]$.

Next, we show that $\frac{\partial P}{\partial y_{i}} (t,\cdot)$
with $i\in\{1,...,h\}$ for all $t\in[0,T]$ exist and are continuous.
In fact, consider an arbitrary but fixed point $y$ and take a
compact set $U\subset R_{c}^{h}$ such that $y$ is in the interior of
$U$. Note that all points in $U$ can be assumed to be bounded by
some positive constant $M$. Thus, by \eq{boundrhoI}, \eq{uniqso},
\eq{yleq} and \eq{usefuline}, we have, for all $s\geq t$,
\begin{eqnarray}
\;\;\;\;\;\;\;\;\left|\frac{\partial}{\partial
y_{i}}\rho(Y^{t,y}(s))\right|
%&\leq&\left(\bar{A}_{1}+\bar{A}_{2}\left\|Y^{t,y}(s)\right\|
%+\bar{A}_{3}\left\|Y^{t,y}(s)\right\|^{2}
%+\bar{A}_{4}\left\|Y^{t,y}(s)\right\|^{3}\right)e^{-\lambda_{i}(s-t)}
%\elabel{dryi}\\
%&\leq&\left(\sum_{i=1}^{4}\bar{A}_{i}\right)
%e^{3\left\|Y^{t,y}(s)\right\|}
%\nonumber\\
%&\leq&\left(\sum_{i=1}^{4}\bar{A}_{i}\right)e^{3
%\sum_{i=1}^{h}Y^{t,y_{i}}_{i}(s)}
%\nonumber\\
%&\leq&\left(\sum_{i=1}^{4}\bar{A}_{i}\right)e^{3\sum_{i=1}^{h}(y_{i}
%+L_{i}(\lambda_{i}s)-L_{i}(\lambda_{i}t))}
%\nonumber\\
\leq\left(\sum_{i=1}^{4}\bar{A}_{i}\right)e^{3hM
+3\sum_{i=1}^{h}L_{i}(\lambda_{i}T)},\elabel{dryi}
\end{eqnarray}
where $Y^{t,y}(s)$ denotes the process with the initial value $y$ at
time $t$.
%and in the first inequality of \eq{dryi}, we have used the fact that
%\begin{eqnarray}
%&&\frac{\partial r(Y^{(t,y)}(s))}{\partial
%y_{i}}=\frac{r(Y^{(t,y)}(s))}{\partial Y^{(t,y)}(s)}\frac{\partial
%Y^{(t,y)}(s)}{\partial y_{i}}. \nonumber
%\end{eqnarray}
Owing to \eq{expintcon} and \eq{basecon}, the function on the
right-hand side of \eq{dryi} is integrable. Thus, it follows from
Theorem 2.27(b) in Folland~\cite{fol:reaana} that the partial
derivative of $\int_{t}^{T}\rho(Y^{t,y}(s))ds$ in terms of $y_{i}$
for each $i\in\{1,...,h\}$ exists. Hence, we have
\begin{eqnarray}
&&\left|\frac{\partial}{\partial y_{i}}
\left(e^{\int_{t}^{T}\rho(Y^{t,y}(s))ds}\right)\right|
%\elabel{intmder}\\
%&=&\left|e^{\int_{t}^{T}\rho(Y^{t,y}(s))ds} \frac{\partial}{\partial
%y_{i}} \left(\int_{t}^{T}\rho(Y^{t,y}(s))ds\right)\right|
%\nonumber\\
%&\leq&T\left(\left(\sum_{i=1}^{4}\bar{A}_{i}\right)e^{3hM
%+3\sum_{i=1}^{h}L_{i}(\lambda_{i}T)}\right)
%\left|e^{\int_{t}^{T}\rho(Y^{t,y}(s))ds}\right|
%\nonumber\\
%&\leq&T\left(\left(\sum_{i=1}^{4}\bar{A}_{i}\right)e^{3hM
%+3\sum_{i=1}^{h}L_{i}(\lambda_{i}T)}\right)
%e^{\int_{t}^{T}(A_{\rho}+B_{\rho}\|Y^{t,y}(s)\|)ds}
%\nonumber\\
\leq T\left(\left(\sum_{i=1}^{4}\bar{A}_{i}\right)
e^{\left(A_{\rho}T+3hM+B_{\rho}\sum_{i=1}^{h}\frac{1}{\lambda_{i}}\right)
+\sum_{i=1}^{h}\left(3+\frac{B_{\rho}}{\lambda_{i}}\right)
L_{i}(\lambda_{i}T)}\right). \elabel{intmder}
\end{eqnarray}
Again, by \eq{expintcon} and \eq{basecon}, we know that the function
on the right-hand side of \eq{intmder} is integrable. Therefore, by
Theorem 2.27(b) in Folland~\cite{fol:reaana}, we can conclude that
$P(t,y)$ is differentiable with respect to $y\in R_{c}^{h}$.
Furthermore, by \eq{uniqso}, \eq{intmder} and the Lebesgue's
dominated convergence theorem, we obtain that the mapping
$(t,y)\rightarrow\frac{\partial P}{\partial y_{i}}(t,y)$ for each
$i\in\{1,...,h\}$ is continuous. Hence, $P(t,y)\in
C^{0,1}\left([0,T]\times R_{c}^{h},R^{1} \right)$.

Third, we prove the square-integrable property \eq{psquarecon} to be
true. In fact, it follows from condition \eq{expintcon} that
$\nu_{i}(\cdot)$ ($i\in\{1,...,h\}$) is a $\sigma$-finite measure
since $\nu_{i}([\epsilon,\infty)) <\infty$ for any $\epsilon>0$. In
addition, it is easy to see that the nonnegative function
$|P(t,Y(t^{-})+z_{i}e_{i})-P(t,Y(t^{-}))|^{2}$ is a measurable one
on the product space $[0,T]\times R_{c}^{h} \times\Omega$. Hence, by
the mean value theorem, \eq{dryi}, \eq{intmder}, the Jensen's
inequality, and the differentiability of $P(t,y)$ in $y$, we have
\begin{eqnarray}
&&E\left[\int_{0}^{T}\int_{z_{i}>0}
|P(t,Y(t^{-})+z_{i}e_{i})-P(t,Y(t^{-}))|^{2}
\nu_{i}(dz_{i})dt\right]
\elabel{provesqu}\\
&\leq&K_{3}K_{4}
\left(e^{(6+\frac{2B_{\rho}}{\lambda_{i}})}\int_{0<z_{i}<1}
z_{i}^{2}\nu_{i}(dz_{i})+\int_{z_{i}\geq
1}\left(e^{(8+\frac{2B_{\rho}}{\lambda_{i}})z_{i}}-1\right)\nu_{i}(dz_{i})
+\int_{z_{i}\geq 1}\nu_{i}(dz_{i})\right) \nonumber\\
&<&\infty, \nonumber
\end{eqnarray}
where $K_{3}$ and $K_{4}$ are some positive constants.
%$\delta_{ij}=1$ if $i=j$ and $\delta_{ij}=0$ otherwise, and in the
%eighth inequality, we have used the similar argument as in
%\eq{usefuline}.
Furthermore, it follows from \eq{expbound}, \eq{yleq}, and
\eq{expintcon} that \eq{ppsisquint} is true.

Fourth, we prove that $P(t,y)$ satisfies the IPDE \eq{rhointdifeq}.
In fact, for each $t\in[0,T)$, it follows from the time-homogeneity
of $Y$ that
\begin{eqnarray}
g(T-t,y)\equiv
E_{0,y}\left[e^{-\int_{0}^{T-t}\rho\left(Y(s)\right)ds}\right]
=E_{t,y}\left[e^{-\int_{t}^{T}\rho\left(Y(s)\right)ds}\right]
=P(t,y). \elabel{gpsit}
\end{eqnarray}
Since $P(t,y)\in C^{0,1}\left([0,T]\times R_{c}^{h}\right)$, it
follows from the It$\hat{o}$'s formula (see, e.g., Theorem 1.14 and
Theorem 1.16 in pages 6-9 of $\emptyset$ksendal and
Sulem~\cite{okssul:appsto}) that, for each fixed $t$,
\begin{eqnarray}
&&g(T-t,Y^{0,y}(l))
\elabel{appitoI}\\
&=&g(T-t,y)-\sum_{i=1}^{h}\lambda_{i}
\int_{0}^{l}Y^{0,y_{i}}_{i}(s^{-})\frac{\partial g} {\partial
y_{i}}(T-t,Y^{0,y}(s^{-}))ds
\nonumber\\
&&+\sum_{i=1}^{h}\int_{0}^{l}
\int_{z_{i}>0}(g(T-t,Y^{0,y}(s^{-})+z_{i}e_{i})
-g(T-t,Y^{0,y}(s^{-})))N_{i}(\lambda_{i}ds,dz_{i}). \nonumber
\end{eqnarray}
Furthermore, let $\hat{g}(t,z_{i},\omega) \equiv
g(T-t,Y^{0,y}(s^{-},\omega)+z_{i}e_{i})
-g(T-t,Y^{0,y}(s^{-}),\omega))$ for each $z_{i}\in(0,\infty)$,
$i\in\{1,...,h\}$ and $\omega\in\Omega$. Then, $\hat{g}$ is $\{{\cal
F}_{t}\}$-predictable. Thus, owing to \eq{psquarecon} (here we need
to use an arbitrary but fixed $y$ to replace $y_{0}$), it follows
from Theorem 4.2.3 in Applebaum~\cite{app:levpro} (or the
explanation in page 61-62 of Ikeda and
Watanabe~\cite{ikewat:stodif}) that the last term in \eq{appitoI} is
a semimartingale. Thus, taking expectations on both sides of
\eq{appitoI}, we get
\begin{eqnarray}
&&\frac{E[g(T-t,Y^{0,y}(l))]-g(T-t,y)}{l}
\nonumber\\
&=&\sum_{i=1}^{h}\frac{\lambda_{i}}{l}
\int_{0}^{l}E\left[Y^{0,y_{i}}_{i}(s^{-}) \frac{\partial g}{\partial
y_{i}} (T-t,Y^{0,y}(s^{-}))\right]ds
\nonumber\\
&&-\sum_{i=1}^{h}\frac{\lambda_{i}}{l}\int_{0}^{l}
\int_{z_{i}>0}E[g(T-t,Y^{0,y}(s^{-})+z_{i}e_{i})
-g(T-t,Y^{0,y}(s^{-}))]\nu_{i}(dz_{i})ds. \nonumber
\end{eqnarray}
Then, by letting $l\downarrow 0$, we know that $P(t,\cdot)$ is in
the domain of the infinitesimal generator of $Y$, which is denoted
by ${\cal A}$, that is,
\begin{eqnarray}
{\cal A}g(T-t,y) &=&\sum_{i=1}^{h}\lambda_{i}y_{i}
\frac{\partial}{\partial y_{i}}g(T-t,y)
\elabel{ggenerator}\\
&&-\sum_{i=1}^{h}\lambda_{i}
\int_{z_{i}>0}(g(T-t,y+z_{i}e_{i})-g(T-t,y))\nu_{i}(dz_{i}).
\nonumber
\end{eqnarray}
Now, by \eq{expbound}, we see that $g(T-t,y)=P(t,Y^{0,y}(l))\in
L^{2}(\Omega,P)$ for each $t\in[0,T)$ and all $l$ in a neighborhood
of zero such that $t-l\leq T$. Thus, we have
\begin{eqnarray}
E_{0,y}[g(T-t,Y(l))]
&=&E_{0,y}\left[E_{0,Y(l)}\left[e^{-\int_{0}^{T-t}
\rho\left(Y(s)\right)ds}\right]\right]
\elabel{hmnkcII}\\
&=&E_{0,y}\left[E_{0,y}\left[\left.e^{-\int_{0}^{T-t}
\rho\left(Y(s+l)\right)ds}\right|{\cal F}_{l}\right]\right]
\nonumber\\
&=&E_{0,y}\left[e^{-\int_{l}^{T-t+l}\rho\left(Y(s))ds\right)}\right]
\nonumber\\
&=&E_{0,y}\left[e^{-\int_{0}^{T-t+l}\rho(Y(s))ds}
e^{\int_{0}^{l}\rho\left(Y(s))ds\right)}\right], \nonumber
\end{eqnarray}
where the second equality in \eq{hmnkcII} follows from the Markov
property of $Y$ (e.g., Proposition 7.9 in
Kallenberg~\cite{kal:foumod}). Then, we have
\begin{eqnarray}
&&\frac{E_{0,y}[g(T-t,Y(l))]-g(T-t,y)}{l}
\elabel{hmnkcIII}\\
%&=&\frac{1}{l}E_{0,y}\left[e^{-\int_{0}^{T-t+l}\rho(Y(s))ds}
%e^{\int_{0}^{l}\rho\left(Y(s))ds\right)}
%-e^{-\int_{0}^{T-t}\rho(Y(s))ds}\right]
%\nonumber\\
%&=&\frac{1}{l}E_{0,y}\left[e^{-\int_{0}^{T-t+l}\rho(Y(s))ds}
%\left(e^{\int_{0}^{l}\rho\left(Y(s))ds\right)}-1\right)\right]\nonumber\\
%&&+\frac{1}{l}E_{0,y}\left[e^{-\int_{0}^{T-t+l}\rho(Y(s))ds}
%-e^{-\int_{0}^{T-t}\rho(Y(s))ds}\right]
%\nonumber\\
&=&\frac{1}{l}E_{0,y}\left[e^{-\int_{0}^{T-t+l}\rho(Y(s))ds}
\left(e^{\int_{0}^{l}\rho(Y(s))ds}-1\right)\right]
+\frac{g(T-t+l,y)-g(T-t,y)}{l}. \nonumber
\end{eqnarray}
Now, by the fundamental theorem of calculus, as $l\downarrow 0$ and
a.s., we have
\begin{eqnarray}
&&e^{-\int_{0}^{T-t+l}\rho(Y^{0,y}(s))ds}\left\{\frac{1}{l}
\left(e^{\int_{0}^{l}\rho(Y^{0,y}(s))ds}-1\right)\right\}
\rightarrow \rho(y)e^{-\int_{0}^{T-t}\rho(Y^{0,y}(s))ds}.
\elabel{calcuint}
\end{eqnarray}
Furthermore, by the mean-value theorem, we have
\begin{eqnarray}
\frac{1}{l}\left|e^{\int_{0}^{l}\rho(Y^{0,y}(s))ds}-1\right|
&\leq&\sup_{l\in[0,T]}\left|\rho(Y^{0,y}(l))
e^{\int_{0}^{l}\rho(Y^{0,y}(s))ds}\right|. \elabel{uniexpb}
\end{eqnarray}
Since the function in the left-hand side of \eq{calcuint} is
uniformly bounded by an integrable function, it follows from the
dominated convergence theorem that the right-derivative of
$g(T-\cdot,y)$ at $t$ exists and satisfies
\begin{eqnarray}
{\cal A}g(T-t,y)=\rho(y)g(T-t)+\frac{\partial g} {\partial
t}(T-t,y). \elabel{eqngp}
\end{eqnarray}
Hence, by \eq{gpsit} and \eq{eqngp}, we know that $P(t,y)$ satisfies
\eq{rhointdifeq}. In addition, we have
%by adopting the similar calculation as
%in \eq{intmder} and \eq{provesqu}, we obtain
\begin{eqnarray}
|P(t,y+z_{i}\delta_{ij})-P(t,y)|
&\leq&K_{5}E\left[e^{3\sum_{j=1}^{h}(2L_{j}(\lambda_{j}T)
+z_{i}\delta_{ij})}\left( e^{\sum_{j=1}^{h}\left(
\frac{B_{\rho}}{\lambda_{j}}(2L_{j}(\lambda_{j}T)+z_{i}\delta_{ij})\right)}
\right)\right]z_{i}.\nonumber
\end{eqnarray}
where $K_{5}$ is some positive constant. Thus, by the Lebesgue's
dominated convergence theorem, we can conclude that
\begin{eqnarray}
\int_{z_{i}>0}|P(t,y+z_{i}e_{i})-P(t,y)| \nu_{i}(dz_{i}) \nonumber
\end{eqnarray}
is continuous in $t$. Therefore, it follows from \eq{rhointdifeq}
that $\frac{\partial P}{\partial t}(t,y)$ is continuous in
$t\in[0,T)$, which implies that $P\in C^{1,1}\left([0,T)\times
R_{c}^{h},R^{1}\right)$. Hence, we complete the proof of
Lemma~\ref{pintdif}. $\Box$
\end{proof}
\begin{lemma}\label{opadp}
Let $O(t)\equiv P(t,Y(t))$ defined in \eq{opyt}. Then, under
conditions {\bf C1}, {\bf C2}, and \eq{expintcon}, $O$ is a
$(0,1]$-valued semimartingale with $O(T)=1$. Furthermore, define
\begin{eqnarray}
&&K\equiv{\cal L}(O)\equiv\left(\frac{1}{O_{-}}\right)\cdot
O\;\;\mbox{with}\;\;K(0)=0\;\;\mbox{and}\;\;O_{-}(t)\equiv O(t^{-}).
\elabel{kloo}
\end{eqnarray}
Then, $K$ is an $\{{\cal F}_{t}\}$-semimartingale and has the
following canonical decomposition
\begin{eqnarray}
&&dK(t)\equiv d{\cal L}(O)(t)=\rho(Y(t^{-}))dt+ \sum_{i=1}^{h}
\int_{z_{i}>0}F(t,z_{i})\tilde{N}_{i}(\lambda_{i}dt,dz_{i}),
\elabel{ksde}
\end{eqnarray}
where, $F(t,z_{i},\omega))$ is defined in \eq{fszy}.
\end{lemma}
\begin{proof}
First, we show that $O$ is an $\{{\cal F}_{t}\}$-semimartingale. In
fact, it follows from the Ito's formula (see, e.g., Theorem 1.14 and
Theorem 1.16 in pages 6-9 of $\emptyset$ksendal and
Sulem~\cite{okssul:appsto}) and Lemma \ref{pintdif} that
\begin{eqnarray}
O(t)&=&P(0,y_{0})+\int_{0}^{t}\rho(Y(s^{-}))P(s,Y(s^{-}))ds
\elabel{lsemin}\\
&&+\sum_{i=1}^{h}\int_{0}^{t}
\int_{z_{i}>0}(P(s,Y(s^{-})+z_{i}e_{i})
-P(s,Y(s^{-})))\tilde{N}_{i}(\lambda_{i}ds,dz_{i}). \nonumber
\end{eqnarray}
Then, by Lemma~\ref{pintdif} and the claim in pages 61-62 of Ikeda
and Watanable~\cite{ikewat:stodif}, we know that the third term in
the right-hand side of \eq{lsemin} is an $\{{\cal
F}_{t}\}$-martingale. Furthermore, by \eq{boundrho} and the similar
proof as used for Lemma~\ref{sprice}, we know that the second term
on the right-hand side of \eq{lsemin} is of finite variation a.s.
Hence, we get that $O$ is an $\{{\cal F}_{t}\}$-semimartingale.
Thus, it follows from \eq{lsemin} and the definition of $K(t)$ that
\eq{ksde} is true.

Second, $M^{K}$ defined as follows is an $\{{\cal
F}_{t}\}$-martingale,
\begin{eqnarray}
M^{K}(t)&=&\sum_{i=1}^{h}\int_{0}^{t}
\int_{z_{i}>0}F(s,z_{i})\tilde{N}_{i}(\lambda_{i}ds,dz_{i}).
\elabel{mkreps}%\\
%M^{D}_{m}(t)&=&\int_{0}^{t}D_{m}(s)\sum_{n=1}^{d}
%\sigma_{mn}(Y(s^{-}))dW_{n}(s).\elabel{mdreps}
\end{eqnarray}
In fact, by the mean-value theorem, \eq{boundrho}, \eq{uniqso}
%\begin{eqnarray}
%&&P(t,Y(t^{-}))\geq e^{-A_{\rho}T-\sum_{i=1}^{h}
%\frac{B_{\rho}}{\lambda_{i}}Y_{i}(t^{-})} E\left[e^{-\sum_{i=1}^{h}
%\frac{B_{\rho}}{\lambda_{i}}L_{i}(\lambda_{i}T)}\right]
%\elabel{pmartingale}
%\end{eqnarray}
\eq{expintcon}, and the fact that $\nu_{i}(\cdot)$
($i\in\{1,...,h\}$) is a $\sigma$-finite measure since
$\nu_{i}([\epsilon,\infty)) <\infty$ for any $\epsilon>0$, we have
\begin{eqnarray}
&&E\left[\int_{0}^{T}\int_{z_{i}>0}
\left|F(t,z_{i})\right|^{2}\nu(dz_{i})dt\right]<\infty.
\elabel{fracpp}
%&\leq&K_{1}\left(e^{(6+\frac{2B_{\rho}}{\lambda_{i}})}\int_{0<z_{i}<1}
%z_{i}^{2}\nu_{i}(dz_{i})+\int_{z_{i}\geq
%1}\left(e^{(8+\frac{2B_{\rho}}{\lambda_{i}})z_{i}}-1\right)\nu_{i}(dz_{i})
%+\int_{z_{i}\geq 1}\nu_{i}(dz_{i})\right)
%\nonumber\\
%&<&\infty\nonumber
\end{eqnarray}
Thus, it follows from \eq{fracpp} and the claims in pages 61-62 of
Ikeda and Watanable~\cite{ikewat:stodif} that $M^{K}$ is an $\{{\cal
F}_{t}\}$-martingale. Therefore, we can conclude that $K$ is an
$\{{\cal F}_{t}\}$-semimartingale. Hence, Lemma~\ref{opadp} is true.
$\Box$
\end{proof}
\begin{lemma}\label{opadpI}
Let $b^{D}$ and $c^{D}$ be the drift and the covariance matrix
processes associated with $D$, $b^{K}$ is the drift process
associated with $K$.
%and $c^{DK}$ is the covariance matrix process between $D$ and $K$.
Then, under conditions {\bf C1}, {\bf C2}, and \eq{expintcon}, we
have
%\begin{eqnarray}
%&&b^{K}=\left(b^{D}+c^{DK}\right)'\left(c^{D}\right)^{-1}
%\left(b^{D}+c^{DK}\right). \elabel{equaoa}
%\end{eqnarray}
\begin{eqnarray}
&&b^{K}=\left(b^{D}\right)'\left(c^{D}\right)^{-1}b^{D}.
\elabel{equaoa}
\end{eqnarray}
Furthermore, the process $a$ defined in \eq{adjustmentp} satisfies
the following relationship,
%\begin{eqnarray}
%&&a\equiv\left(c^{D}\right)^{-1} \left(b^{D}+c^{DK}\right).
%\elabel{adjustp}
%\end{eqnarray}
\begin{eqnarray}
&&a\equiv\left(c^{D}\right)^{-1}b^{D}. \elabel{adjustp}
\end{eqnarray}
\end{lemma}
\begin{proof}
First of all, it follows from Lemma~\ref{disprice} and
Lemma~\ref{opadp} that
\begin{eqnarray}
b^{D}(t)&=&(D_{1}(t)(b_{1}(Y(t^{-}))-r),...,D_{d}(t)(b_{d}(Y(t^{-}))-r))',
\elabel{ddeI}\\
c^{D}(t)&=&\mbox{diag}(D(t))\left(\sigma(Y(t^{-}))\sigma(Y(t^{-}))'\right)
\mbox{diag}(D(t)),
\elabel{ddeII}\\
b^{K}(t)&=&O^{-1}(t)b^{O}(t)=\rho(Y(t^{-})).
\elabel{ddeIII}
%c^{K}(t)&=&c^{KD}(t)=c^{DK}(t)=c^{KY}=0 \elabel{ddeIV}
\end{eqnarray}
Then, by simple calculations, we know that \eq{equaoa} and
\eq{adjustp} are true. Hence, we complete the proof of
Lemma~\ref{opadpI}. $\Box$
\end{proof}

For convenience, we will use $C_{ij}^{D}\equiv[D_{i},D_{j}]$ to
denote the co-quadratic variation processes with $i,j\in\{1,...,d\}$
for the process $D$ and write interchangeably $c^{D_{i}D_{j}}\equiv
c_{ij}^{D}$ and $c^{D_{i}}=c_{ii}^{D}$. Furthermore, similar
notations are also used for other processes related in the following
discussions.
\begin{lemma}\label{zfpp}
Under conditions {\bf C1}, {\bf C2}, and \eq{expintcon}, $\hat{Z}$
is an $\{{\cal F}_{t}\}$-and $P$-martingale.
\end{lemma}
\begin{proof}
First, we show that $a\in L(D)$. In fact, it follows from  the
condition {\bf C1}, \eq{uniqso}, and \eq{dsde} that
$\|Y(t^{-}))\|\geq\min\{y_{i0}e^{-\lambda_{i}T},i=1,...,h\}>0$ for
any $t\in[0,T]$. Then, for $m,n\in\{1,...,d\}$, we have
\begin{eqnarray}
\bar{\rho}(Y(t^{-}))&\equiv&\sum_{m=1}^{d}\left(B(Y(t^{-}))'
\left(\sigma(Y(t^{-}))\sigma(Y(t^{-}))'\right)^{-1}\right)^{2}_{m}
\sum_{n=1}^{d}\sigma_{mn}^{2}(Y(t^{-})) \elabel{barrho}\\
&\leq&C_{\bar{\rho}}
+\frac{B_{b}^{2}B_{\sigma}}{b_{\sigma}^{2}}\|Y(t^{-})\|,
\nonumber
\end{eqnarray}
where $C_{\bar{\rho}}$ is some positive constant. Thus, it follows
from the Kunita-Watanable inequality (e.g., Theorem 25 in page 69 of
Protter~\cite{pro:stoint}) that
\begin{eqnarray}
E\left[\sum_{m=1}^{d}\sum_{n=1}^{d}\left|
\int_{0}^{T}a_{m}(t)a_{n}(t)d\left[M_{m}^{D},M_{n}^{D}\right](t)
\right|\right]&\leq&
d^{2}E\left[\int_{0}^{T}\bar{\rho}(Y(t^{-}))dt\right]
\elabel{mnmmI}\\
&<&\infty,\nonumber
\end{eqnarray}
where $a_{m}$ and $M^{D}_{m}$ with $m\in\{1,...,d\}$ are the $m$th
components of $a$ and $M^{D}$ respectively.
%and
%\begin{eqnarray}
%\bar{\rho}(Y(t^{-}))&=&\sum_{m=1}^{d}\left(B(Y(t^{-}))'
%\left(\sigma(Y(t^{-}))\sigma(Y(t^{-}))'\right)^{-1}\right)^{2}_{m}
%\sum_{n=1}^{d}\sigma_{mn}^{2}(Y(t^{-})) \elabel{barrho}\\
%&\leq&C_{\bar{\rho}}
%+\frac{B_{b}^{2}B_{\sigma}}{b_{\sigma}^{2}}\|Y(t^{-})\| \nonumber
%\end{eqnarray}
Furthermore, it follows from \eq{dsde} that
\begin{eqnarray}
&&E\left[\sum_{m=1}^{d}\int_{0}^{T}a_{m}(t)D_{m}(t)
B_{m}(Y(t^{-}))dt\right]=E\left[\int_{0}^{T}\rho(Y(t^{-}))dt\right]
<\infty. \elabel{mnmmII}
\end{eqnarray}
Then, by \eq{mnmmI}-\eq{mnmmII}, Definition 6.17 of page 207,
Definition 4.3 of page 180, Definition 6.12 of page 206, and
Definition 2.6 of page 76 in Jacod and
Shiryaev~\cite{jacshi:limthe}, we know that $a\in L(D)$. Thus,
$(a\cdot D)(T)$ is well defined. %through \eq{stointud}.

In addition, it follows from Theorem 4.5(a) in page 180 of Jacod and
Shiryaev~\cite{jacshi:limthe} that, for each $u\in L(D)$, we have,
\begin{eqnarray}
&&(u\cdot D)(t)=\lim_{k\rightarrow\infty}
\sum_{i=1}^{d}\int_{0}^{t}u_{i}(s)I_{\{\|u(s)\|\leq
k\}}dM^{D}_{i}(s) +\sum_{i=1}^{d}\int_{0}^{t}u_{i}(s)dB^{D}_{i}(s),
\elabel{stointud}
\end{eqnarray}
where the limit in the first term on the right-hand side of
\eq{stointud} corresponds to the convergence in probability
uniformly on every compact set of $[0,T]$. Therefore, by \eq{dsde},
\eq{expintcon}, \eq{dsdeb}-\eq{mdreps}, \eq{stointud}, and the
Lebesgue dominated convergence theorem, we know that
\begin{eqnarray}
(a\cdot D)(T)=\sum_{m=1}^{d}\int_{0}^{T}a_{m}(t)dD_{m}(t).
\elabel{explicitad}
\end{eqnarray}

Now, it follows from Lemma~\ref{opadp} that $O$ is a semimartingale.
Thus, it follows from Conditions {\bf C1}, {\bf C2} and
\eq{explicitad} that $(a\cdot D)$ is also a semimartingale. Then, by
Corollary 8.7(b) and equation 8.19 in pages 135-138 of Jacod and
Shiryaev~\cite{jacshi:limthe}, we have that
\begin{eqnarray}
\hat{Z}(t)&=&{\cal E}(K-(a\cdot D)-[K,(a\cdot D)])(t)
\elabel{jhatz}\\
&=&{\cal E}\left(M^{K}-(a\cdot M^{D})+(b^{K}-a'b^{D})\cdot
A\right)(t)
\nonumber\\
&=&{\cal E}(G)(t),\nonumber
\end{eqnarray}
where the second equality follows from the facts that $A(t)=t$,
$K(0)=0$ and the independence among driving Brownian motions and
L\'evy processes. The third equality follows from
Lemma~\ref{opadpI}. Furthermore, $M^{K}$ and $M^{D}$ are given by
\eq{mkreps} and \eq{mdreps}, which are $\{{\cal
F}_{t}\}$-martingales. Hence,
\begin{eqnarray}
&&G\equiv M^{K}-\left(a\cdot M^{D}\right)\elabel{grepsn}
\end{eqnarray}
is also an $\{{\cal F}_{t}\}$-martingale. Thus, it follows from
Theorem 4.61 in page 59 of Jacod and Shiryaev~\cite{jacshi:limthe}
that $\hat{Z}$ is an $\{{\cal F}_{t}\}$-local martingale.

Second, we prove that $\hat{Z}$ is of class (D), i.e., the set of
random variables
\begin{eqnarray}
&&\{\hat{Z}(\tau),\tau\;\mbox{is finite valued}\;\{{\cal
F}_{t}\}-\mbox{stopping times}\} \nonumber
\end{eqnarray} is
uniformly integrable (e.g., Definition 1.46 in page 11 of Jacod and
Shiryaev~\cite{jacshi:limthe}).

In fact, consider an arbitrary finite-valued $\{{\cal
F}_{t}\}$-stopping time $\tau\leq T$ and an arbitrary constant
$\gamma>0$. Then, we have
\begin{eqnarray}
&&E\left[\left|\hat{Z}(\tau)\right|
I_{\{|\hat{Z}(\tau)|\geq\gamma\}}\right]
\leq\frac{1}{P(0,y_{0})}\left(E\left[\left({\cal E}(-a\cdot
D)(\tau)\right)^{2}\right]\right)^{1/2}
\left(P\{|\hat{Z}(\tau)|\geq\gamma\}\right)^{1/2},\elabel{classd}
\end{eqnarray}
where we have used the facts that $0<O(\cdot)\leq 1$ and $D$ is
continuous. Furthermore, let
\begin{eqnarray}
&&U_{1}(t)=\int_{0}^{t}\rho(Y(s^{-}))ds, \elabel{mmi}\\
&&U_{2}(t)=\sum_{n=1}^{d}\int_{0}^{t}\bar{B}_{n}(Y(s^{-}))dW_{n}(s),
\elabel{mmii}
\end{eqnarray}
where $\bar{B}(Y(s^{-}))$ is defined in \eq{barbii}. Hence,
$U_{2}(t)$ is a continuous $\{{\cal F}_{t}\}$-martingale. Thus,
\begin{eqnarray}
E\left[\left({\cal E}\left(-a\cdot D\right)(\tau)\right)^{2}\right]
&=& E\left[e^{\left(-2(U_{1}(\tau)+U_{2}(\tau))
-[U_{1}+U_{2},U_{1}+U_{2}](\tau) \right)}\right]\elabel{caletau}\\
&\leq&E\left[e^{-2U_{2}(\tau)}\right]
\nonumber\\
&\leq&\left(E\left[e^{8[U_{2},U_{2}](T)}\right]\right)^{\frac{1}{2}}
\nonumber\\
&<&\infty,\nonumber
\end{eqnarray}
where the third inequality follows from the optional sampling
theorem, the fact that $e^{-2U_{2}(t)}$ is a submartingale by the
Jensen's inequality, and Theorem 39 in page 138 of
Protter~\cite{pro:stoint}. The last inequality follows from
conditions {\bf C1}-{\bf C2}. Therefore, it follows from
\eq{caletau} that
$\sup_{\tau}E\left[\left|\hat{Z}(\tau)\right|\right]\leq K_{1}$,
where $K_{1}$ is some positive constant. Thus, by the Markov's
inequality, we have that
\begin{eqnarray}
P\{|\hat{Z}(\tau)|\geq\gamma\}\leq\frac{K_{1}}{\gamma}\rightarrow
0\;\;\mbox{as}\;\gamma\rightarrow\infty\elabel{mhatz}
\end{eqnarray}
for all stopping time $\tau\leq T$. Therefore, it follows from
\eq{classd}-\eq{mhatz} that $\hat{Z}$ is of class (D). Hence, it
follows from \eq{jhatz} and Proposition 1.47(c) in page 12 of Jacod
and Shiryaev~\cite{jacshi:limthe} that $\hat{Z}$ is a uniformly
integrable $\{{\cal F}_{t}\}$- and $P$-martingale. $\Box$
\end{proof}

\begin{lemma}\label{emmp}
Under conditions {\bf C1}, {\bf C2}, and \eq{expintcon}, $Q^{*}$ is
an equivalent martingale measure.
\end{lemma}
\begin{proof}
First, we use $P_{t}$ to denote the restriction of $P$ to ${\cal
F}_{t}$ for each $t\in[0,T]$. Then, we define
$dQ^{*}_{t}\equiv\hat{Z}(t)dP_{t}$ and $dQ^{*}\equiv\hat{Z}(T)dP$.
Owing to \eq{pexactsolution}-\eq{densitym}, we know that
$\hat{Z}(t)>0$ for each $t\in[0,T]$.
%Then it follows from the definition of Radon-Nikodym derivative
%that, for any event $B\in{\cal F}_{t}$, we have
%\begin{eqnarray}
%&&Q^{*}_{t}(B)=\int_{B}\hat{Z}(t)P_{t}(d\omega)=\int_{B}
%E\left[\hat{Z}(T)|{\cal
%F}_{t}\right]P_{t}(d\omega)=\int_{B}\hat{Z}(T)P_{t}(d\omega)=Q(B)
%\elabel{eqequiv}
%\end{eqnarray}
%where in the second equation of \eq{eqequiv}, we have used the fact
%that $\hat{Z}$ is a $\{{\cal F}_{t}\}$- and $P$-martingale.
%Thus, $Q^{*}_{t}$ is the restriction of $Q^{*}$ to ${\cal F}_{t}$
%and is equivalent to $P_{t}$ for each $t\in[0,T]$. Therefore, it
%follows from Theorem 3.4 in page 166 of Jacod and
%Shiryaev~\cite{jacshi:limthe}, we conclude that $Q^{*}$ is
%equivalent to $P$ with the density process $\hat{Z}$.
Furthermore, note that $\hat{Z}$ is a $\{{\cal F}_{t}\}$- and
$P$-martingale. Hence, it follows from the discussion in page 166 of
Jacod and Shiryaev~\cite{jacshi:limthe} that $Q^{*}$ is equivalent
to $P$ with the density process $\hat{Z}$.

Next, we show that $D$ is an $Q^{*}$-martingale. In fact, since $D$
is an $P$-semimartingale with the decomposition given in \eq{dsde},
it follows from Girsanov-Meyer Theorem (e.g., Theorem 35 in page 132
of Protter~\cite{pro:stoint}) that $D$ is also an
$Q^{*}$-semimartingale with the decomposition $D=\tilde{D}+\bar{D}$.
The process $\bar{D}$ is an $Q^{*}$-finite variation process. For
each $m\in\{1,...,d\}$,
\begin{eqnarray}
\tilde{D}_{m}(t)&=&M^{D}_{m}(t)-\int_{0}^{t}\frac{1}{\hat{Z}(s)}
d\left[\hat{Z},M^{D}_{m}\right](s) \elabel{tilded}\\
%&=&M^{D}_{m}(t)-\sum_{n=1}^{d}\int_{0}^{t}
%\frac{D_{m}(s)\sigma_{mn}(Y(s^{-}))}{\hat{Z}(s)}
%d\left[\hat{Z},W_{n}\right]^{c}(s) \nonumber\\
&=&M^{D}_{m}(t)-\sum_{n=1}^{d}\int_{0}^{t}
\frac{D_{m}(s)\sigma_{mn}(Y(s^{-}))}{\hat{Z}(s)O_{0}} d\left[O{\cal
E}(-a\cdot D),W_{n}\right]^{c}(s) \nonumber\\
&=&M^{D}_{m}(t)-\sum_{n=1}^{d}\int_{0}^{t}
\frac{D_{m}(s)\sigma_{mn}(Y(s^{-}))}{\hat{Z}(s)O_{0}}{\cal
E}(-a\cdot D(s))\left(d\left[O,W_{n}\right]^{c}(s)\right.\nonumber\\
&&\left.+O(s)d\left[U,W_{n}\right]^{c}(s)
+\frac{1}{2}d[[O,U]^{c},W_{n}]^{c}(s)\right)\nonumber\\
&=&M^{D}_{m}(t)-\sum_{n=1}^{d}\int_{0}^{t}
D_{m}(s)\sigma_{mn}(Y(s^{-}))d\left[U,W_{n}\right]^{c}(s)
\nonumber\\
&=&M^{D}_{m}(t)+\sum_{r=1}^{d}\sum_{n=1}^{d}\int_{0}^{t}
D_{m}(s)\sigma_{mn}(Y(s^{-}))a_{r}(s)d\left[D_{r},W_{n}\right]^{c}(s)
\nonumber\\
&=&M^{D}_{m}(t)+\int_{0}^{t}\sum_{n=1}^{d}
D_{m}(s)\sigma_{mn}(Y(s^{-}))\bar{B}_{n}(Y(s^{-}))ds,
\nonumber
%&=&M^{D}_{m}(t)+\int_{0}^{t}D_{m}(s)(\sigma_{m1}(Y(s^{-})),...,
%\sigma_{md}(Y(s^{-})))\sigma(Y(s^{-}))' \nonumber\\
%&&\left(\sigma(Y(s^{-})) \sigma(Y(s^{-}))'\right)^{-1} B(Y(s^{-}))ds
%\nonumber
\end{eqnarray}
where $\bar{B}_{n}(Y(s^{-}))$ is defined in \eq{barbiiI}. The second
equality in \eq{tilded} follows from Theorem 29 in page 75 of
Protter~\cite{pro:stoint}, the proof of Corollary in page 83 of
Protter~\cite{pro:stoint}, the fact that $W$ is continuous, Theorem
4.52 in page 55 of Jacod and Shiryaev~\cite{jacshi:limthe}, and the
explanation in page 70 of Protter~\cite{pro:stoint}.
%(or Theorem 4.47(c) in page 52 of Jacod and Shiryaev~\cite{jacshi:limthe}).
The third equality in \eq{tilded} follows from the Ito's formula for
multi-dimensional semimartingales (e.g., Theorem 33 in pages 81-82
of Protter~\cite{pro:stoint}), and the associated function $f$ is
taken to be $f(O,U)=Oe^{U}$. Furthermore, $a_{r}$ is the $r$th
component of $a$, and $U$ is defined by $U(t)\equiv-a\cdot
D(t)-\frac{1}{2}[a\cdot D,a\cdot D](t)$. Thus, we have
\begin{eqnarray}
\bar{D}(t)=D(t)-\tilde{D}(t)=\bar{s}\equiv(s_{1},...,s_{d})'\;\;
\mbox{or}\;\;D(t)=\tilde{D}(t)+\bar{s}, \elabel{ddequal}
\end{eqnarray}
where $s_{i}$ for each $i\in\{1,...,d\}$ is the initial price as
given in \eq{stockassetm}.

Therefore, to show that $D$ is an $Q^{*}$-martingale, it suffices to
show that $\tilde{D}$ is an $Q^{*}$-martingale. More precisely, by
the last equation in the proof of Theorem 35 in pages 132-133 of
Protter~\cite{pro:stoint}, we have that
\begin{eqnarray}
&&\tilde{D}_{m}(t)=\left(M^{D}_{m}(t)-\frac{1}{\hat{Z}(t)}
\left[\hat{Z},M_{m}^{D}\right](t)\right)+\int_{0}^{t}
\left[\hat{Z},M_{m}^{D}\right](s^{-})
d\left[\frac{1}{\hat{Z}}\right](s).\elabel{gmtdec}
\end{eqnarray}
Then, we can show that the both terms on the right-hand side of
\eq{gmtdec} are $Q^{*}$-martingales.

For the first term on the right-hand side of \eq{gmtdec}, it follows
from integration by parts (e.g., equations (*) and (**) in page 132
of Protter~\cite{pro:stoint}), the Ito's formula (e.g., Theorem 1.14
and Theorem 1.16 in pages 6-9 of $\emptyset$ksendal and
Sulem~\cite{okssul:appsto}), and Lemma~\ref{pintdif} that
\begin{eqnarray}
&&\left(M^{D}_{m}(t)-\frac{1}{\hat{Z}(t)}
\left[\hat{Z},M_{m}^{D}\right](t)\right)\hat{Z}(t) \elabel{pzmart}\\
&=&\int_{0}^{t}\hat{Z}(s^{-})dM_{m}^{D}(s)
+\int_{0}^{t}M_{m}^{D}(s)d\hat{Z}(s) \nonumber\\
&=&\int_{0}^{t}\hat{Z}(s^{-})dM_{m}^{D}(s)-\sum_{n=1}^{d}
\int_{0}^{t}M_{m}^{D}(s)\hat{Z}(s^{-})\bar{B}_{n}(Y(s^{-}))
dW_{n}(s)\nonumber\\
&&+\sum_{i=1}^{h}\int_{0}^{t}\int_{z_{i}>0} \frac{M_{m}^{D}(s){\cal
E}((-a\cdot D)(s))}{O_{0}}
(P(s,Y(s^{-})+z_{i}e_{i})-P(s,Y(s^{-})))\tilde{N}_{i}(\lambda_{i}ds,dz_{i}),
\nonumber
\end{eqnarray}
where $\bar{B}_{n}(Y(s^{-}))$ is defined in \eq{barbiiI}. The second
equality follows from \eq{mkreps}-\eq{mdreps} and the fact that
\begin{eqnarray}
d\hat{Z}(t)=\hat{Z}(t^{-})dG(t) \elabel{stoexp}
\end{eqnarray}
owing to \eq{jhatz}-\eq{grepsn}, the definition of Dol\'eans-Dade
exponential, and Theorem 37 in pages 84-85 of
Protter~\cite{pro:stoint}.

Then, we can show that each of the three terms on the right-hand
side of \eq{pzmart} is an $Q^{*}$-martingale.

The claim that the first term on the right-hand side of \eq{pzmart}
is a $Q^{*}$-martingale can be proved as follows. First, it follows
from the similar argument as used in \eq{mkls} that $M^{D}$ is a
square integrable $P$-martingale. Second, by the Tonelli's Theorem
(e.g., Theorem 20 in page 309 of Royden \cite{roy:reaana}) and the
H$\ddot{o}$lder's inequality, we have
\begin{eqnarray}
&&E\left[\int_{0}^{T}\hat{Z}^{2}(s)d\left[M^{D}_{m},M^{D}_{m}\right](s)
\right]\elabel{ezmdm}\\
%&=&E\left[\int_{0}^{T}\hat{Z}^{2}(s)D_{m}^{2}(s)
%\left(\sum_{n=1}^{d}\sigma_{mn}(Y(s^{-}))\right)^{2}ds\right]
%\nonumber\\
%&\leq&\frac{d}{O^{2}_{0}}\left(E\left[K_{1}+K_{2}\|L(\lambda
%T)\|\right]\right)^{\frac{1}{2}}
%\left(E\left[\left(\int_{0}^{T}P^{2}(s,Y(s^{-}))\left({\cal
%E}(-a\cdot
%D)(s)\right)^{2}D_{m}^{2}(s)\right)^{2}\right]ds\right)^{\frac{1}{2}}
%\nonumber\\
&\leq&\bar{K}\int_{0}^{T}\left(E\left[O^{8}(s)\right]\right)^{1/2}
\left(E\left[\left({\cal E}(-a\cdot
D)(s)\right)^{16}\right]\right)^{1/4}\left(E\left[D_{m}^{16}(s)\right]
\right)^{1/4}ds\nonumber\\
&<&\infty,\nonumber
\end{eqnarray}
where $\bar{K}$ is some positive constant. The last inequality in
\eq{ezmdm} follows from the similar arguments as in \eq{caletau} and
\eq{dmeight}. Thus, it follows from Theorem 4.40(b) in page 48 of
Jacod and Shiryaev~\cite{jacshi:limthe} that the first term on the
right-hand side of \eq{pzmart} is an $\{{\cal F}_{t}\}$- and
$P$-martingale.

The claim that the second term on the right-hand side of \eq{pzmart}
is an $Q^{*}$-martingale can be proved as follows. It follows from
\eq{dmeight} and Exercise 3.25 in page 163 of Karatzas and
Shreve~\cite{karshr:bromot} that
\begin{eqnarray}
&&E\left[\int_{0}^{t}\left(M_{m}^{D}(s)\right)^{16}ds\right]<\infty.
\elabel{mkls}
%&=&\int_{0}^{t}E\left[\left(\int_{0}^{s}D_{m}(u)\sum_{n=1}^{d}
%\sigma_{mn}(Y(u^{-}))dW_{n}(u)\right)^{16}\right]ds \nonumber\\
%&\leq&d^{15}\int_{0}^{t}\sum_{n=1}^{d}E\left[\left(\int_{0}^{s}D_{m}(u)
%\sigma_{mn}(Y(u^{-}))dW_{n}(u)\right)^{16}\right]ds\nonumber\\
%&\leq&8\cdot
%15^{8}d^{15}\int_{0}^{t}\sum_{n=1}^{d}s^{7}\int_{0}^{s}E\left[|D_{m}(u)
%\sigma_{mn}(Y(u^{-}))|^{16}\right]duds\nonumber\\
%&\leq&8\cdot 15^{8}d^{15}T^{7}\left(E\left[(K_{1}+K_{2}\|L(\lambda
%T)\|)^{16}\right]\right)^{\frac{1}{2}}\int_{0}^{t}\int_{0}^{s}
%\left(E\left[|D_{m}(u)|^{32}\right)^{\frac{1}{2}}\right]duds\nonumber\\
%&<&\infty\nonumber
\end{eqnarray}
%where $K_{1}$ and $K_{2}$ are some positive constants.
Then, by \eq{mkls}, the H$\ddot{o}$lder's inequality and the similar
method as used in \eq{ezmdm}, we know that the second term on the
right-hand side of \eq{pzmart} is an $\{{\cal F}_{t}\}$- and
$P$-martingale.

The claim that the third term on the right-hand side of \eq{pzmart}
is an $Q^{*}$-martingale can be proved as follows. It follows from
the Tonelli's Theorem (e.g., Theorem 20 in page 309 of Royden
\cite{roy:reaana}) that
\begin{eqnarray}
&&E\left[\int_{0}^{T}\int_{z_{i}>0}\frac{|M_{m}^{D}(t)|}{O_{0}}
\left|(P(t,Y(t^{-})+z_{i}e_{i})-P(t,Y(t^{-}))){\cal E}((-a\cdot
D)(t))\right|\nu_{i}(dz_{i})dt\right]
\elabel{loneintpe}\\
%&\leq&\int_{0}^{T}\int_{z_{i}>0}
%E\left[\frac{|M_{m}^{D}(t)|}{O_{0}}{\cal E}(-a\cdot
%D)(t)\sup_{\xi(Y(t^{-}))\in[0,z_{i}]}\left|\frac{\partial
%P(t,Y(t^{-}) +\xi(Y(t^{-}))e_{i})}{\partial
%y_{i}}\right|\right]z_{i}\nu_{i}(dz_{i})dt
%\nonumber\\
%&\leq&\frac{1}{O_{0}}\int_{0}^{T}\int_{z_{i}>0}
%\left(E\left[\left(M_{m}^{D}(t)\right)^{4}\right]\right)^{\frac{1}{4}}
%\left(E\left[\left({\cal E}((-a\cdot
%D)(t))\right)^{4}\right]\right)^{\frac{1}{4}}
%\nonumber\\
%&&\left(E\left[\sup_{\xi(Y(t^{-}))\in[0,z_{i}]}\left|\frac{\partial
%P(t,Y(t^{-}) +\xi(Y(t^{-}))e_{i})}{\partial
%y_{i}}\right|^{2}\right]\right)^{\frac{1}{2}}z_{i}\nu_{i}(dz_{i})dt
%\nonumber\\
%&\leq&\bar{K}_{1}\int_{0}^{T}\int_{z_{i}>0}
%\left(\int_{0}^{T}\left(E\left[\left|D_{m}(s)\right|^{4}\right)^{\frac{1}{2}}
%\right]ds\right)^{\frac{1}{4}}\nonumber\\
%&&\left(E\left[e^{\int_{0}^{T}\left(8\rho(Y(s^{-}))+4d\bar{\rho}(Y(s^{-}))
%\right)ds}\right]\right)^{1/4}\left(E\left[e^{64d\int_{0}^{T}
%\bar{\rho}(Y(s^{-}))ds}\right]\right)^{1/4}\nonumber\\
%&&\left(E\left[\sup_{\xi(Y(t^{-}))\in[0,z_{i}]}\left|\frac{\partial
%P(t,Y(t^{-}) +\xi(Y(t^{-}))e_{i})}{\partial
%y_{i}}\right|^{2}\right]\right)^{\frac{1}{2}}z_{i}\nu_{i}(dz_{i})dt
%\nonumber\\
&\leq&K_{1}\left(E\left[\int_{0}^{T}\int_{z_{i}>0}
\sup_{\xi(Y(t^{-}))\in[0,z_{i}]}\left|\frac{\partial P(t,Y(t^{-})
+\xi(Y(t^{-}))e_{i})}{\partial
y_{i}}\right|^{2}z_{i}\nu_{i}(dz_{i})dt\right]\right)^{\frac{1}{2}}
\nonumber\\
&<&\infty,\nonumber
\end{eqnarray}
%where $\bar{K}_{1}$ and $K_{1}$ are some positive constants,
where $K_{1}$ is some positive constant.
%$\bar{\rho}(Y(s^{-}))$ is defined in \eq{barrho},
The inequalities in \eq{loneintpe} follow from the similar proofs as
used in \eq{caletau}, \eq{mkls}, the H$\ddot{o}$lder's inequality,
the proof of \eq{provesqu}, and the fact that
\begin{eqnarray}
\int_{z_{i}>0}z_{i}\nu(dz_{i})
&\leq&\int_{0<z_{i}<1}z_{i}\nu_{i}(dz_{i})+
\int_{z_{i}\geq 1}z_{i}\nu_{i}(dz_{i})<\infty.\nonumber\\
&=&\int_{0<z_{i}<1}z_{i}\nu_{i}(dz_{i})+\int_{z_{i}\geq
1}\left(e^{z_{i}}-1\right)+\int_{z_{i}\geq
1}\nu_{i}(dz_{i})\nonumber\\
&<&\infty.\nonumber
\end{eqnarray}
Then, it follows from \eq{loneintpe} and the argument in pages 61-62
in Ikeda and Watanable~\cite{ikewat:stodif} that the third term on
the right-hand side of \eq{pzmart} is also an $\{{\cal F}_{t}\}$-
and $P$-martingale.

Therefore, by summarizing the discussions for the three terms on the
right-hand side of \eq{pzmart}, we know that the process given by
\eq{pzmart}, is an $\{{\cal F}_{t}\}$- and $P$-martingale. Moreover,
by applying Proposition 3.8(a) in page 168 of Jacod and
Shiryaev~\cite{jacshi:limthe}, we can conclude that the first term
on the right-hand side of \eq{gmtdec} is an $Q^{*}$-martingale.

For the second term on the right-hand side of \eq{gmtdec}, we can
show that it is also an $\{{\cal F}_{t}\}$- and $Q^{*}$-martingale.
In fact, since $\hat{Z}$ is a density process of $Q^{*}$ in terms of
$P$ and $\left(\frac{1}{\hat{Z}}\right)\hat{Z}=1$ (that is an
$P$-martingale), it follows from Proposition 3.8(a) in page 168 of
Jacod and Shiryaev~\cite{jacshi:limthe} that $\frac{1}{\hat{Z}}$ is
an $Q^{*}$-martingale. Furthermore, it follows from the Ito's
formula (e.g., Theorem 32 in page 78 of Protter~\cite{pro:stoint}),
\eq{stoexp} and the calculation of $d\hat{Z}(t)$ in the last
equality in \eq{pzmart} that
\begin{eqnarray}
d\left(\frac{1}{\hat{Z}(t)}\right)
%&=&-\frac{1}{\hat{Z}^{2}(t^{-})}
%d\hat{Z}(t)+\frac{1}{\hat{Z}^{3}(t^{-})}
%d\left[\hat{Z},\hat{Z}\right]^{c}(t)
%+\left(-\frac{1}{\hat{Z}(t)\hat{Z}(t^{-})}+
%\frac{1}{\hat{Z}^{2}(t^{-})}\right)
%\left(\hat{Z}(t)-\hat{Z}(t^{-})\right) \nonumber\\
&=&\frac{1}{\hat{Z}(t^{-})}
\sum_{n=1}^{d}\left(\bar{B}_{n}(Y(t^{-}))\right)^{2}dt
-\frac{1}{\hat{Z}(t^{-})}
\sum_{n=1}^{d}\bar{B}_{n}(Y(t^{-}))dW_{n}(t)
\elabel{reversez}\\
&&-\sum_{i=1}^{h}\int_{z_{i}>0}\frac{F(t,z_{i})}{\hat{Z}(t)}
\tilde{N}_{i}(\lambda_{i}dz_{i},dt), \nonumber
\end{eqnarray}
where $\bar{B}(Y(t^{-}))$ is defined in \eq{barbiiI}. Thus, it
follows from \eq{reversez} that $\frac{1}{\hat{Z}}$ is squarely
integrable under $Q^{*}$, i.e.,
\begin{eqnarray}
E_{Q^{*}}\left[\left(\frac{1}{\hat{Z}(t)}\right)^{2}\right]
&\leq&E_{Q^{*}}\left[\sup_{0\leq s\leq
T}\frac{1}{\hat{Z}^{2}(s)}\right]
\elabel{oneoverhatz}\\
&\leq&4E_{Q^{*}}\left[\frac{1}{\hat{Z}^{2}(T)}\right]\nonumber\\
&\leq&4\left(E\left[\hat{Z}^{2}(T)\right]\right)^{\frac{1}{2}}
\left(E\left[\frac{1}{\hat{Z}^{4}(T)}\right]\right)^{\frac{1}{2}}
\nonumber\\
&=&\frac{4}{O_{0}}\left(E\left[({\cal E}((-a\cdot D)(T)))^{2}\right]
\right)^{\frac{1}{2}}\left(E\left[\frac{1}{\left({\cal E}((-a\cdot
D)(T)) \right)^{4}}\right]\right)^{\frac{1}{2}}
\nonumber\\
&<&\infty, \nonumber
\end{eqnarray}
where the second inequality in \eq{oneoverhatz} follows from the
Doob's martingale inequality (e.g., Theorem 2.1.5 in page 74 of
Applebaum~\cite{app:levpro}) since $\frac{1}{\hat{Z}}$ is an
$Q^{*}$-martingale. The last inequality of \eq{oneoverhatz} follows
from the similar argument as in \eq{caletau}.

Therefore, to show that the second term on the right-hand side of
\eq{gmtdec} is an $Q^{*}$-martingale, it suffices to show that the
following expectation under $Q^{*}$ is finite owing to
\eq{oneoverhatz} and Theorem 4.40(b) in page 48 of Jacod and
Shiryaev~\cite{jacshi:limthe},
\begin{eqnarray}
&&E_{Q^{*}}\left[\int_{0}^{T}\left(\left[\hat{Z},M_{m}^{D}\right](s^{-})
\right)^{2}d\left[\frac{1}{\hat{Z}},\frac{1}{\hat{Z}}\right](s)\right]
\elabel{decsec}\\
&=&E_{Q^{*}}\left[\int_{0}^{T}
\left(\left[\hat{Z},M_{m}^{D}\right]^{c}(s^{-})
\frac{1}{\hat{Z}(s^{-})}\right)^{2}
\sum_{n=1}^{d}\left(\bar{B}_{n}(Y(s^{-}))\right)^{2}ds\right]
\nonumber\\
&&+\sum_{i=1}^{h}E_{Q^{*}}\left[\int_{0}^{T}\int_{z_{i}>0}
\left(\left[\hat{Z},M_{m}^{D}\right]^{c}(s^{-})
\frac{F(s,z_{i})}{\hat{Z}(s)}\right)^{2}
\lambda_{i}\nu_{i}(dz_{i})ds\right].\nonumber
\end{eqnarray}
The first term on the right-hand side of \eq{decsec} is finite since
\begin{eqnarray}
&&E_{Q^{*}}\left[\int_{0}^{T}\frac{1}{\hat{Z}^{2}(s^{-})}
\left(\left[\hat{Z},M_{m}^{D}\right]^{c}(s^{-})\right)^{2}
\bar{\rho}(Y(s^{-}))ds\right]
\elabel{ldecsec}\\
%&=&E_{Q^{*}}\left[\int_{0}^{T}\frac{\bar{\rho}(Y(s^{-}))}{\hat{Z}^{2}(s^{-})}
%\left(\int_{0}^{s}\hat{Z}(u^{-})D_{m}(u)
%\sum_{n=1}^{d}\sigma_{mn}(Y(u^{-}))du\right)^{2}ds\right]
%\nonumber\\
%&\leq&TE_{Q^{*}}\left[\int_{0}^{T}\frac{\bar{\rho}(Y(s^{-}))}{\hat{Z}^{2}(s^{-})}
%\int_{0}^{s}\hat{Z}^{2}(u^{-})D^{2}_{m}(u)
%\left(\sum_{n=1}^{d}\sigma_{mn}(Y(u^{-}))\right)^{2} duds\right]
%\nonumber\\
%&\leq&dTE_{Q^{*}}\left[\int_{0}^{T}
%\frac{\bar{\rho}(Y(s^{-}))}{\hat{Z}^{2}(s^{-})}
%\int_{0}^{s}\hat{Z}^{2}(u^{-})D^{2}_{m}(u)
%\sum_{n=1}^{d}\sigma^{2}_{mn}(Y(u^{-}))duds\right]
%\nonumber\\
%&\leq&dT^{2}E_{Q^{*}}\left[\sup_{0\leq s\leq
%T}\frac{1}{\hat{Z}^{2}(s)}\phi(\|L(\lambda T)\|)\sup_{0\leq s\leq
%T}\hat{Z}^{2}(s)\int_{0}^{T}D_{m}^{2}(s)ds\right]
%\nonumber\\
%&\leq&dT^{2}\left(E_{Q^{*}}\left[\left(\sup_{0\leq s\leq
%T}\frac{1}{\hat{Z}^{4}(s)}\right)\right]\right)^{\frac{1}{2}}
%\left(E\left[\left(\phi(\|L(\lambda
%T))\|)\right)^{4}\right]\right)^{\frac{1}{4}}
%\nonumber\\
%&&\left(E\left[\sup_{0\leq s\leq
%T}|\hat{Z}(s)|^{20}\right]\right)^{\frac{1}{8}}
%\left(E\left[\left(\int_{0}^{T}D^{2}_{m}(s)ds\right)^{8}\right]
%\right)^{\frac{1}{8}}\nonumber\\
%&\leq&K_{1}\left(\frac{4}{3}E\left[\frac{1}{|\hat{Z}(T)|^{3}}\right]
%\right)^{\frac{1}{2}}\left(\frac{20}{19}E\left[|\hat{Z}(T)|^{20}\right]
%\right)^{\frac{1}{8}}
%\left(\int_{0}^{T}E\left[D^{8}_{m}(s)\right]ds\right)^{\frac{1}{8}}
%\nonumber\\
&\leq&K_{1}\left(\frac{4}{3}E\left[\left(\frac{1}{{\cal E}(-a\cdot
D)(T)}\right)^{3}\right]\right)^{\frac{1}{2}}
\left(\frac{20}{19}E\left[\left({\cal E}(-a\cdot
D)(T)\right)^{20}\right]\right)^{\frac{1}{8}}
\left(\int_{0}^{T}E\left[D^{8}_{m}(s)\right]ds\right)^{\frac{1}{8}}
\nonumber\\
&<&\infty,\nonumber
\end{eqnarray}
where $K_{1}$ is some positive constant.
%$\bar{\rho}(Y(s^{-}))$ is
%defined in \eq{barrho}, $\phi(\cdot)$ is some nonnegative polynomial
%function in terms of $\|L(\lambda T)\|$, which is obtained by
%applying condition {\bf C1};
The first inequality in \eq{ldecsec} follows from the Doob's
martingale inequality (e.g., page 74 of
Applebaum~\cite{app:levpro}). The second inequality in \eq{ldecsec}
follows from the similar arguments as in \eq{classd} and
\eq{dmeight}. Similarly, the second term on the right-hand side of
\eq{decsec} is also finite, which can be proved along the line of
the discussion as in \eq{ldecsec}.
%In fact, along the line of the discussion as in
%\eq{ldecsec}, we have that
%\begin{eqnarray}
%&&E_{Q^{*}}\left[\int_{0}^{T}\int_{z_{i}>0}
%\left|\left[\hat{Z},M_{m}^{D}\right]^{c}(s^{-})
%\frac{F(s,z_{i})}{\hat{Z}(s)}\right|^{2}
%\nu_{i}(dz_{i})ds\right]<\infty,\nonumber
%&\leq&E_{Q^{*}}\left[\phi_{1}(\|L(\lambda T\|) \sup_{0\leq s\leq
%T}\frac{1}{\hat{Z}^{2}(s)} \sup_{0\leq s\leq
%T}\hat{Z}^{2}(s)\int_{0}^{T}D_{m}^{2}(s)ds
%\int_{0}^{T}\int_{z_{i}>0}F^{2}(s,z_{i})\nu_{i}(dz_{i})ds\right]
%\nonumber\\
%&\leq&\left(E\left[\left(\int_{0}^{T}\int_{z_{i}>0}F^{2}(s,z_{i})
%\nu_{i}(dz_{i})dt\right)^{2}\right]\right)^{\frac{1}{2}}
%\left(\frac{11}{10}E\left[\hat{Z}^{11}(T)\right]\right)^{\frac{1}{4}}
%\left(\frac{8}{7}E\left[\frac{1}{\hat{Z}^{7}(T)}\right]
%\right)^{\frac{1}{8}}\nonumber\\
%&&\left(E\left[(\phi_{1}(\|L(\lambda
%T)\|))^{16}\right]\right)^{\frac{1}{16}}
%\left(E\left[\left(\int_{0}^{T}D_{m}^{2}(t)dt\right)^{16}\right]
%\right)^{\frac{1}{16}}
%\nonumber\\
%&<&\infty\nonumber
%\end{eqnarray}
%where $\phi_{1}(\|L(\lambda T)\|)$ is a positive polynomial function
%of $\|L(\lambda T)\|$, and in the last inequality, we have used the
%similar method as in \eq{provesqu}, \eq{fracpp} and \eq{loneintpe}.

Thus, it follows from the finiteness of \eq{decsec} that the second
term on the right-hand side of \eq{gmtdec} is an $Q^{*}$-martingale.
Therefore, by combining this fact with \eq{gmtdec} and \eq{pzmart},
we know that $D=\tilde{D}+\bar{s}$ displayed in \eq{ddequal} is an
$Q^{*}$-martingale (i.e. $Q^{*}$ is an equivalent martingale
measure). Finally, by applying the similar discussion as used in
\eq{ezmdm}, we conclude that $\frac{dQ^{*}}{dP}\in L^{2}(P)$, which
implies that $Q^{*}\in{\cal U}_{2}^{e}(D)$. $\Box$
\end{proof}

\begin{lemma}
Under conditions {\bf C1}, {\bf C2}, and \eq{expintcon}, $Q^{*}$ is
the VOMM.
\end{lemma}
\begin{proof}
It suffices to justify that all conditions stated in Theorem 3.25 of
C\u{e}rn\'y and Kallsen~\cite{cerkal:strgen} are satisfied. First of
all, for any stopping time $\tau$, we can show that
\begin{eqnarray}
&&u^{\tau}(t)\equiv a(t)I_{(\tau,T]}(t){\cal
E}\left((-aI_{(\tau,T])})\cdot D\right)(t^{-})\in\bar{\Theta}(D).
\elabel{utaud}
\end{eqnarray}
In fact, it follows from the proof of Lemma~\ref{emmp} that ${\cal
U}_{2}^{e}(D)$ is nonempty. Furthermore, since $D$ is a continuous
$P$-semimartingale, it is sufficient to prove that the three
equivalent conditions stated in Theorem 2.1 of C\u{e}rn\'y and
Kallsen~\cite{cerkal:meavar} are satisfied for \eq{utaud}, which can
be done by tedious computations similarly as before. In addition, we
can show that $O{\cal E}((-aI_{(\tau,T]})\cdot D)$ is of class
($D$). Therefore, by combining this claim with Lemma~\ref{opadpI},
\eq{utaud}, and Theorem 3.25 in C\u{e}rn\'y and
Kallsen~\cite{cerkal:strgen}, we know that $O$ and $a$ are the
opportunity and adjustment processes in the sense defined Section 3
of~\cite{cerkal:strgen}. Thus, it follows from Proposition 3.13 in
C\u{e}rn\'y and Kallsen~\cite{cerkal:strgen} that $Q^{*}$ is the
VOMM. Hence, we complete the proof of Proposition~\ref{equimar}.
$\Box$
\end{proof}

\subsection{The Unique Existence of Solution to A Type of BSDEs}
Consider the following $q$-dimensional BSDE with jumps and a
terminal condition $H$
\begin{eqnarray}
V(t)&=&H-\int_{t}^{T}g\left(s,V(s^{-}),\bar{V}(s),\tilde{V}(s,\cdot),
Y(s^{-})\right)ds-\int_{t}^{T}\sum_{i=1}^{d}\bar{V}_{i}(s)dW_{i}(s)
\elabel{gbsde}\\
&&-\int_{t}^{T}\sum_{i=1}^{h}\int_{z_{i}>0}\tilde{V}_{i}(s,z_{i})
\tilde{N}_{i}(\lambda_{i}ds,dz_{i}), \nonumber
\end{eqnarray}
where $H\in L^{2}_{{\cal F}_{T}}(\Omega,R^{q},P)$,
$\bar{V}=\left(\bar{V}_{1},...,\bar{V}_{d}\right)\in R^{q\times d}$,
$\tilde{V}=\left(\tilde{V}_{1},...,\tilde{V}_{h}\right)\in
R^{q\times h}$, $g$ is a random function: $[0,T]\times R^{q}\times
R^{q\times d}\times L^{2}_{\nu}(R^{h}_{+},R^{q\times h})\times
R^{h}\times\Omega\rightarrow R^{h}$ and
\begin{eqnarray}
&&L^{2}_{\nu}(R^{h}_{+},R^{q\times h})\equiv\left\{\tilde{v}:
R^{h}_{+}\rightarrow R^{q\times h},\sum_{i=1}^{h}\int_{z_{i}>0}
\left|\tilde{v}_{i}(z_{i})\right|^{2}\nu_{i}(dz_{i})<\infty\right\}.
\elabel{lvrqh}
\end{eqnarray}
Furthermore, for any $\tilde{v}\in L^{2}_{\nu}(R^{h}_{+},R^{q\times
h})$, the associated norm is defined by
\begin{eqnarray}
\|\tilde{v}\|_{\nu}\equiv\left(\sum_{i=1}^{h}\int_{z_{i}>0}
\left|\tilde{v}_{i}(z_{i})\right|^{2}\lambda_{i}\nu_{i}(dz_{i})
\right)^{\frac{1}{2}}. \elabel{vnorm}
\end{eqnarray}
\begin{proposition}\label{bsdey}
Replacing $H\in L^{4}_{{\cal F}_{T}}(\Omega,R,P)$ by $H\in
L^{2}_{{\cal F}_{T}}(\Omega,R,P)$ in Assumption~\ref{hstopasump}.
Supposing that $g(t,v,\bar{v},\tilde{v},Y(t^{-}))$ is $\{{\cal
F}_{t}\}$-adapted for any given $(v,\bar{v},\tilde{v})\in
R^{q}\times R^{q\times d}\times L^{2}_{\nu}(R^{h}_{+},R^{q\times
h})$ with
\begin{eqnarray}
g(\cdot,0,0,,0,Y(\cdot^{-}))\in L^{2}_{{\cal
F}}\left([0,T],R^{q}\right)\elabel{blipic}
\end{eqnarray}
such that
\begin{eqnarray}
&&\left\|\left(g\left(t,v,\bar{v},\tilde{v},Y(t^{-})\right)-g\left(t,u,
\bar{u},\tilde{u},Y(t^{-})\right)\right)I_{\{t\leq\tau_{n}\}}\right\|
\elabel{blipschitz}\\
&\leq& K_{n}\left(\|u-v\|+\|\bar{u}-\bar{v}\|
+\|\tilde{u}-\tilde{v}\|_{\nu}\right)\nonumber
\end{eqnarray}
for any $(u,\bar{u},\tilde{u})$ and $(v,\bar{v},\tilde{v})\in
R^{q}\times R^{q\times d}\times L^{2}_{\nu}(R^{h}_{+},R^{q\times
h})$, where $K_{n}$ depending on $n$ are positive constants. Then,
the BSDE in \eq{gbsde} has a unique solution
\begin{eqnarray}
(V,\bar{V},\tilde{V})\in L^{2}_{{\cal F}}([0,T],R^{q},P)\times
L^{2}_{{\cal F},p}([0,T],R^{q\times d},P)\times
L^{2}_{p}([0,T],R^{q\times h},P), \elabel{buniso}
\end{eqnarray}
where $V$ is a c\`adl\`ag process. The uniqueness is in the sense:
if there exists another solution $(U,\bar{U},\tilde{U})$ as
required, then,
\begin{eqnarray}
E\left[\int_{0}^{T}\left(\|U(t)-V(t)\|^{2}+\|\bar{U}(t)
-\bar{V}(t)\|^{2}+\|\tilde{U}(t,\cdot)-\tilde{V}(t,\cdot)\|_{\nu}^{2}\right)dt
\right]=0. \elabel{uniquesbI}
\end{eqnarray}
\end{proposition}
\begin{proof}
First, for each $n\in\{1,2,...\}$, we define
\begin{eqnarray}
&&\tau_{n}\equiv\inf\{t>0,\|L(\lambda t)\|> n\}.
\elabel{stopsequence}
\end{eqnarray}
Then, it follows from Theorem 3 in page 4 of
Protter~\cite{pro:stoint} and condition \eq{expintcon} that
$\{\tau_{n}\}$ is a sequence of nondecreasing $\{{\cal
F}_{t}\}$-stopping times and satisfies $\tau_{n}\rightarrow\infty$
a.s. as $n\rightarrow\infty$ since
\begin{eqnarray}
P\{\tau_{n}\leq t\}=P\{\|L(\lambda t)\|>
n\}\leq\frac{E\left[\|L(\lambda t)\|^{2}\right]}{n^{2}}\rightarrow 0
\nonumber
\end{eqnarray}
as $n\rightarrow\infty$ for any given $t\in[0,\infty)$, where we
have used \eq{expintcon}, \eq{basecon}, \eq{usefuline}, and the fact
that $L(\lambda t)$ is a $h$-dimensional nonnegative and
nondecreasing c\`adl\`ag process.

Second, for each $n$, consider the following BSDE with a random
terminal time $\sigma_{n}\equiv T\wedge\tau_{n}$ and a terminal
condition $H_{\tau_{n}}$,
\begin{eqnarray}
V(t)&=&H_{\tau_{n}}-\int_{t\wedge\sigma_{n}}^{\sigma_{n}}
g\left(s,V(s^{-}),\bar{V}(s),\tilde{V}(s,\cdot),Y(s^{-})\right)ds
\elabel{taugbsde}\\
&&-\int_{t\wedge\sigma_{n}}^{\sigma_{n}}
\sum_{i=1}^{d}\bar{V}_{i}(s)dW_{i}(s)
-\int_{t\wedge\sigma_{n}}^{\sigma_{n}}
\sum_{i=1}^{h}\int_{z_{i}>0}\tilde{V}_{i}(s,z_{i})
\tilde{N}_{i}(\lambda_{i}ds,dz_{i}). \nonumber
\end{eqnarray}
Then, by slightly generalizing the discussion as in Yong and
Zhou~\cite{yonzho:stocon} and Tang and Li~\cite{tanli:neccon} (see
also El Karoui {\em et al.}~\cite{karpen:bacsto},
Situ~\cite{sit:solbac}, Yin and Mao~\cite{yinmao:adasol} for related
discussions), we know that $\eq{taugbsde}$ has a unique adapted
solution as required over $[0,\sigma_{n}]$.

Third, for each $n\in\{1,2,...\}$, let
$\Omega_{n}=\{\omega\in\Omega:\sigma_{n}(\omega)=T\}$. Since
$\sigma_{n}$ is a sequence of nondecreasing stopping times and
$\sigma_{n}\rightarrow T$ a.s. as $n\rightarrow\infty$, we have that
$\Omega=\cup_{n=1}^{\infty}\Omega_{n}$ and
$\Omega_{l}\subseteq\Omega_{n}$ whenever $l\leq n$. Now, we use
$\Pi^{n}(t,z)\equiv(V^{n}(t),\bar{V}^{n}(t),\tilde{V}^{n}(t,z))$ for
$t\leq\sigma_{n}$ and $z\in R^{h}_{+}$ to denote the unique solution
to \eq{taugbsde} for each $n$. Since
$H_{\tau_{n}}(\omega)=H(\omega)$ for all
$\omega\in\{\omega:\tau_{n}(\omega)\geq T\}$, we know that
$\Pi^{n}(t,z)=\Pi^{n-1}(t,z)=...=\Pi^{l}(t,z)$ for all
$t\leq\sigma_{l}(\omega)$, a.s. $\omega\in\Omega_{l}$ and any $z\in
R^{h}_{+}$. By the continuity of probability, we know that, for any
given $\epsilon>0$, there exists a sufficiently large $n_{0}>0$ such
that $P\{\Omega_{n}\}>1-\epsilon$ when $n>n_{0}$. Thus, for any
given $\delta>0$ and for all $n,l>n_{0}$, we have
\begin{eqnarray}
&&P\left\{\sup_{0\leq t\leq T,z\in
R^{h}_{+}}\left\|\Pi^{n}(t\wedge\sigma_{n},z)
-\Pi^{l}(t\wedge\sigma_{l},z)
\right\|>\delta\right\}<\epsilon,\nonumber
\end{eqnarray}
that is, $\{\Pi^{n}(\cdot\wedge\sigma_{n},\cdot),n\in\{1,2,...\}\}$
is uniformly Cauchy in probability. Thus, it is uniformly convergent
in probability to a process $\Pi=\{\Pi(t,z),t\in[0,T],z\in
R^{h}_{+}\}$. Therefore, we can extract a subsequence from
$\{\Pi^{n}(\cdot\wedge\sigma_{n},\cdot),n\in\{1,2,...\}\}$ such that
the convergence holds uniformly a.s. Hence, we can conclude that
$\Pi$ is a solution to \eq{gbsde} and have all the properties as
stated in the proposition. Furthermore, assume that
$\Pi'=\{\Pi'(t,z),t\in[0,T],z\in R^{h}_{+}\}$ is another solution to
\eq{gbsde}. Then, we can conclude that, for all $n\geq l$,
$\Pi'(t,z,\omega)=\Pi^{l}(t,z,\omega)$ for all $t\in[0,T]$, $z\in
R^{h}_{+}$, and almost all $\omega\in\Omega_{l}$. In fact, if the
claim fails to be true for some $n\geq l$, define
$\Pi^{''}_{n}(t,z,\omega)=\Pi'(t,z,\omega)$ for
$\omega\in\Omega_{l}$ and
$\Pi^{''}_{n}(t,z,\omega)=\Pi^{n}(t,z,\omega)$ for
$\omega\in\Omega_{l}^{c}$. Then, $\Pi''_{n}$ and $\Pi^{n}$ are
distinct solutions to \eq{taugbsde} with the same terminal condition
$H_{\tau_{n}}$, which contradicts the uniqueness of solution to
\eq{taugbsde}. Then, $P\{\Pi(t,z)=\Pi'(t,z)\;\mbox{for
all}\;t\in[0,T],z\in R^{h}_{+}\}=1$ follows from a straightforward
limiting argument as above. Furthermore, by applying the similar
argument as used for Definition 2.4 and its associated remark in
page 57 of Ikeda and watanabe~\cite{ikewat:stodif}, we know that
$\Pi$ is the unique solution to \eq{gbsde} (interested readers are
also referred to pages 309-310 of Applebaum~\cite{app:levpro} for
some related discussion). Hence, we complete the proof of
Proposition~\ref{bsdey}. $\Box$
\end{proof}

\subsection{Remaining Proof of Theorem~\ref{opthedge}}

First of all, by the H$\ddot{o}$lder's inequality and the similar
calculation as for \eq{caletau}, we have that
\begin{eqnarray}
E\left[\left(H{\cal E}(-a\cdot D)(T)\right)^{2}\right]
&\leq&\left(E\left[H^{4}\right]\right)^{\frac{1}{2}}
\left(E\left[\left({\cal E}(-a\cdot
D)(T)\right)^{4}\right]\right)^{\frac{1}{2}}
\elabel{hsquarein}\\
&<&\infty.\nonumber
\end{eqnarray}
Thus, it follows from the Jensen's inequality that the process
$X=\{X(t),t\in[0,T]\}$ with
\begin{eqnarray}
X(t)&\equiv&E\left[H{\cal E}(-a\cdot D)(T)|{\cal F}_{t}\right]
\elabel{klpbayes}
\end{eqnarray}
is a square-integrable martingale. Thus, by the Martingale
representation theorem (e.g., Lemma 2.3 in Tang and
Li~\cite{tanli:neccon}), we have
\begin{eqnarray}
X(t)=X(0)+\sum_{j=1}^{d}\int_{0}^{t}\bar{X}_{j}(s)dW_{j}(s)
+\sum_{i=1}^{h}\int_{0}^{t}\int_{z_{i}>0}\tilde{X}_{i}(s,z_{i})
\tilde{N}_{i}(\lambda_{i}ds,dz_{i}) \elabel{xhmarrep}
\end{eqnarray}
with $\bar{X}=(\bar{X}_{1},...,\bar{X}_{d})'\in L_{{\cal
F},p}^{2}([0,T],R^{d},P)$ and
$\tilde{X}=(\tilde{X}_{1},...,\tilde{X}_{h})'\in
L^{2}_{p}([0,T],R^{h},P)$.
%Furthermore, let
%\begin{eqnarray}
%V(t)=E_{Q^{*}}\left[H|{\cal F}_{t}\right]. \elabel{defineV}
%\end{eqnarray}
Furthermore, it follows from the Bayes' rule (e.g., Lemma 8.6.2 in
page 160 of $\emptyset$ksendal~\cite{oks:stodif}) and
Proposition~\ref{equimar} that
\begin{eqnarray}
X(t)=O_{0}E\left[H\hat{Z}(T)|{\cal
F}_{t}\right]=O_{0}\hat{Z}(t)V(t), \elabel{pbayes}
%&=&O_{0}E_{Q^{*}}\left[H|{\cal F}_{t}\right]E\left[\hat{Z}(T)|{\cal
%F}_{t}\right]\nonumber\\
%&=&O_{0}\hat{Z}(t)V(t).\nonumber
\end{eqnarray}
where $V(t)$ is defined in \eq{defineV}. Thus, by the integration by
parts formula (e.g., Corollary 2 in page 68 of
Protter~\cite{pro:stoint}), and \eq{xhmarrep}-\eq{pbayes}, we have
\begin{eqnarray}
&&dV(t)\elabel{pbsdev}\\
&=&\frac{1}{O_{0}}\left(X(t^{-})d\left(\frac{1}{\hat{Z}(t)}\right)
+\frac{1}{\hat{Z}(t^{-})}dX(t)+d\left[X,\frac{1}{\hat{Z}}\right](t)\right)
\nonumber\\
&=&g(t,V(t^{-}),\bar{V}(t),\tilde{V}(t,\cdot),Y(t^{-})))dt\nonumber\\
&&+\sum_{i=1}^{d}\bar{V}_{i}(t)dW_{i}(t)+\sum_{i=1}^{h}\int_{z_{i}>0}
\tilde{V}_{i}(t,z_{i})\tilde{N}_{i}(\lambda_{i}dz_{i},dt),
\nonumber
\end{eqnarray}
where $g$ is defined in \eq{mgvvv} and
\begin{eqnarray}
&&\bar{V}_{i}(t)=-V(t^{-})\bar{B}_{i}(Y(t^{-}))
+\frac{\bar{X}_{i}(t)}{O_{0}\hat{Z}(t^{-})},\;\;\;\;\;\;\;\;\;\;\;\;
\mbox{for}\;i=1,...,d,
\nonumber\\
&&\tilde{V}_{i}(t,z_{i})=-V(t^{-})F(t,z_{i})\bar{Z}(t)
+\frac{\tilde{X}_{i}(t,z_{i})}{O_{0}\hat{Z}(t^{-})},\;\;\;\;\;
\mbox{for}\;i=1,...,h
\nonumber
\end{eqnarray}
with $\bar{Z}$ given by \eq{barbii}. Hence, by \eq{pbsdev}, we know
that $V$ satisfies the BSDE \eq{xbsde}. %Furthermore, owing to the
%Corollary in page 8 of Protter~\cite{pro:stoint}, $V$ can be viewed
%as a c\`adl\`ag process.

Next, we check that $g(t,v,\bar{v},\tilde{v},Y(t^{-}))$ defined in
\eq{mgvvv} satisfies the conditions as stated in
Proposition~\ref{bsdey}. In fact, from \eq{mgvvv}, we see that
$g(t,v,\bar{v},\tilde{v},Y(t^{-}))$ is ${{\cal F}_{t}}$-adapted for
any given $(v,\bar{v},\tilde{v})\in R\times R^{1\times d}\times
L^{2}_{\nu}(R^{h}_{+},R^{1\times h})$ with
$g(t,0,0,0,Y(t^{-}))\equiv 0\in L^{2}_{{\cal F}}([0,T],R,P)$,
Furthermore, for the sequence of nondecreasing stopping times
$\{\tau_{n},n=1,2,...\}$ as defined in \eq{stopsequence},
%Then we can conclude that $\tau_{n}\rightarrow\infty$ a.s. when
%$n\rightarrow\infty$ since
%\begin{eqnarray}
%P\{\tau_{n}\leq t\}=P\{\|L(\lambda t)\|>
%n\}\leq\frac{E\left[\|L(\lambda t)\|^{2}\right]}{n^{2}}\rightarrow 0
%\nonumber
%\end{eqnarray}
%as $n\rightarrow\infty$ for any given $t\in[0,\infty)$, where we
%have used \eq{expintcon}, \eq{basecon}, \eq{usefuline}, and the fact
%that $L(\lambda t)$ is nonnegative and nondecreasing in $t$.
%Thus, by \eq{boundrho}, \eq{yleq} and \eq{disequine}, we have
we have
\begin{eqnarray}
&&\left|\bar{Z}(t)\right|I_{\{t\leq\tau_{n}\}}
\leq\bar{K}_{n}e^{\sum_{i=1}^{h}\frac{2B_{\rho}}{\lambda_{i}}
\|L(\lambda t)\|}I_{\{t\leq\tau_{n}\}} \leq\tilde{K}_{n},
\nonumber
\end{eqnarray}
where $\bar{K}_{n}$ and $\tilde{K}_{n}$ are positive constants
depending on $n$. In addition, it follows from the proof of
\eq{fracpp} that
\begin{eqnarray}
&&\left(\int_{z_{i}>0}(F(t,z_{i}))^{2}\nu_{i}(dz_{i})\right)
I_{\{t\leq\tau_{n}\}}
\leq\bar{L}e^{\sum_{i=1}^{h}(6+\frac{4B_{\rho}}{\lambda_{i}})\|
L(\lambda t)\|}I_{\{t\leq\sigma_{n}\}} \leq\tilde{L}_{n},\nonumber
\end{eqnarray}
where $\bar{L}$ is some positive constant and $\tilde{L}_{n}$ is a
positive constant depending on $n$. Therefore, for any
$(u,\bar{u},\tilde{u})$, $(v,\bar{v},\tilde{v})\in R\times
R^{1\times d}\times L^{2}_{\nu}(R^{h}_{+},R^{1\times h})$, we have
\begin{eqnarray}
&&\left\|(g(t,u,\bar{u},\tilde{u},Y(t^{-}))-g(t,v,\bar{v},
\tilde{v},Y(t^{-})))I_{\{t\leq\tau_{n}\}}\right\|
\nonumber\\
&\leq&h\tilde{K}^{2}_{n} \tilde{L}_{n}\|u-v\|
+\left\|\bar{u}-\bar{v}\right\|\left(\frac{1}{2}(\rho(Y(t^{-}))+d)
\right)I_{\{t\leq\tau_{n}\}}+h\lambda_{i}\tilde{K}_{n}
\left(\tilde{L}_{n}\right)^{\frac{1}{2}}
\|\tilde{u}-\tilde{v}\|_{\nu}\nonumber\\
&\leq&K_{n}\left(\|u-v\|+\|\bar{u}-\bar{v}\|+\|\tilde{u}-\tilde{v}\|_{\nu}
\right),\nonumber
\end{eqnarray}
where $K_{n}$ is some positive constant depending on $n$ and in the
last inequality, we have used \eq{boundrho}. Thus, all conditions
stated in Proposition~\ref{bsdey} are satisfied, which implies that
\eq{xbsde} has a unique adapted solution.

Now, for each $t\in[0,T]$ and
$B^{K}(t)=\int_{0}^{t}\rho(Y(s^{-}))ds$, we define the density
process
\begin{eqnarray}
Z^{P^{*}}(t)\equiv\frac{O(t)}{O_{0}{\cal E}(B^{K})(t)}.
\elabel{zpequiv}
\end{eqnarray}
Then, the corresponding probability $P^{*}\sim P$. Thus, it is the
opportunity-neutral probability measure in the sense of Definition
3.16 in C\u{e}rn\'y and Kallsen~\cite{cerkal:strgen}. Furthermore,
by Corollary 8.7(b) and equation (8.19) in pages 135-138 of Jacod
and Shiryaev~\cite{jacshi:limthe}, we can rewrite $Z^{P^{*}}$ in
\eq{zpequiv} as
\begin{eqnarray}
Z^{P^{*}}(t)&=&{\cal E}(K)(t){\cal E}\left(-B^{K}\right)(t)={\cal
E}\left(M^{K}\right)(t) \elabel{rzpe}
\end{eqnarray}
for each $t\in[0,T]$, where $K$ is defined in \eq{kloo} and $M^{K}$
is defined in \eq{mkreps}. Then, by a similar method as used in the
proof of Proposition~\ref{equimar}(2),  we know that $Z^{P^{*}}$ is
a bounded positive martingale. Thus, for each pair of
$i,j\in\{1,...,d\}$ and $t\in[0,T]$, we have
\begin{eqnarray}
&&\langle D_{i},D_{j}\rangle^{P^{*}}(t)=[D_{i},D_{j}]^{P^{*}}(t)
=[D_{i},D_{j}](t)=\int_{0}^{t}\tilde{c}^{D^{*}}_{ij}(s)ds,
\elabel{ddcij}
\end{eqnarray}
where the first equality in \eq{ddcij} is owing to the continuity of
$D$, Theorem 5.52 in page 55 of Jacod and
Shiryaev~\cite{jacshi:limthe}, Theorem 4.47(c) in page 52 of Jacod
and Shiryaev~\cite{jacshi:limthe}, the equivalence between $P^{*}$
and $P$, and Girsanov-Meyer Theorem in page 132 of
Protter~\cite{pro:stoint}. The second equality follows from Theorem
4.47(a) in page 52 of Jacod and Shiryaev~\cite{jacshi:limthe} since
$Z^{P^{*}}$ is bounded and Girsanov-Meyer Theorem in page 132 of
Protter~\cite{pro:stoint}. Furthermore, $\tilde{c}^{D^{*}}_{ij}$ in
the last equality is defined in \eq{xicbI}.

Now, note that $D$ is continuous. Then, by Theorem 4.52 in page 55
of Jacod and Shiryaev~\cite{jacshi:limthe} (or the proof of
Corollary in page 83 of Protter~\cite{pro:stoint}), we know that
$[D_{i},V](t)$ and $[D_{i},V]^{c}(t)$ for each $i\in\{1,...,d\}$
under $P$ or $P^{*}$ have the same compensator. Hence, we have
\begin{eqnarray}
&&\langle D_{i},V\rangle^{P^{*}}(t)=\left(\langle
D_{i},V\rangle^{c}\right)^{P^{*}}(t)=\left([D_{i},V]^{c}
\right)^{P^{*}}(t)=[D_{i},V]^{c}(t)=\int_{0}^{t}
\tilde{c}^{DV^{*}}_{i}(s)ds \elabel{dvij},
\end{eqnarray}
where $\tilde{c}^{DV^{*}}_{i}$ is defined in \eq{xicbII}. The last
equality of \eq{dvij} follows from Theorem 4.47(a) in page 52 of
Jacod and Shiryaev~\cite{jacshi:limthe} and the fact that
\begin{eqnarray}
V(t)
&=&V(0)+\int_{0}^{t}g(s,V(s^{-}),\bar{V}(s),\tilde{V}(s,\cdot),Y(s^{-}))ds
\nonumber\\
&&+\int_{0}^{t}\sum_{i=1}^{d}\bar{V}_{i}(s)dW_{i}(s)
+\int_{0}^{t}\sum_{i=1}^{h}\int_{z_{i}>0}
\tilde{V}_{i}(s,z_{i})\tilde{N}_{i}(\lambda_{i}dz_{i},ds).\nonumber
\nonumber
\end{eqnarray}
Then, it follows from \eq{ddcij}-\eq{dvij}, Definition 4.6, and
equation (4.8) in C\u{e}rn\'y and Kallsen~\cite{cerkal:strgen} that
\eq{xicb} is true.

Finally, the unique existence of solution to \eq{geqn} is owing to
Theorem 6.8 in Jacod~\cite{jac:calsto} and the proofs of Lemma 4.9
and Theorem 4.10 in C\u{e}rn\'y and Kallsen~\cite{cerkal:strgen}.
Thus, by Theorem 4.10 in C\u{e}rn\'y and
Kallsen~\cite{cerkal:strgen}, we know that the mean-variance hedge
strategy is given by \eq{optimalh}. Hence, we complete the proof of
Theorem~\ref{opthedge}. $\Box$

\section{Conclusion}\label{concl}

In this paper, we prove the global risk optimality of the hedging
strategy explicitly constructed for an incomplete financial market.
Owing to the discussions in Pigorsch and
Stelzer~\cite{pigste:defsta} and references therein, our discussion
in this paper can be extended to the cases that the external risk
factors in \eq{sdeou} are correlated in certain manners. For the
simplicity of notation, we keep the presentation of the paper in the
current way. Furthermore, our study in this paper establishes the
connection between our financial system and existing general
semimartingale based study in C\u{e}rn\'y and
Kallsen~\cite{cerkal:strgen} since we can overcome the difficulties
in C\u{e}rn\'y and Kallsen~\cite{cerkal:strgen} by explicitly
constructing the process $N$ and the VOMM $Q^{*}$. In addition, our
objective and discussion in this paper are different from the recent
study of Jeanblanc {\em et al.}~\cite{jeaman:meavar} since the
authors in Jeanblanc {\em et al.}~\cite{jeaman:meavar} did not aim
to derive any concrete expression. Nevertheless, interested readers
may make an attempt to extend the study in Jeanblanc {\em et
al.}~\cite{jeaman:meavar} and apply it to our financial market model
to construct the corresponding explicit results. Finally, unlike the
studies in Hubalek {\em et al.}~\cite{hubkal:varopt} and Kallsen and
Vierthauer~\cite{kalvie:quahed}, our option $H$ is generally related
to a multivariate terminal function and hence a BSDE involved
approach is employed. Interested readers may take an attempt to
study whether the Laplace transform related method developed in
Hubalek {\em et al.}~\cite{hubkal:varopt} and Kallsen and
Vierthauer~\cite{kalvie:quahed} for single-variate terminal function
can be extended to our general multivariate case.

\end{document}